\documentclass[11pt,a4paper]{article}

\RequirePackage{amsmath}
\RequirePackage{amssymb,latexsym}
\RequirePackage{amsthm}
\RequirePackage[hang,flushmargin,bottom]{footmisc}
\RequirePackage[margin=0.35in,format=hang,font=small,labelfont=bf]{caption}
\RequirePackage{needspace}
\RequirePackage{graphicx,tikz}
\RequirePackage{mpatt}
\RequirePackage{stmaryrd}
\RequirePackage[T1]{fontenc}
\RequirePackage[sc]{mathpazo}
\RequirePackage[colorlinks=true,urlcolor=blue,linkcolor=blue,citecolor=blue]{hyperref}

\hypersetup{pdfstartview={XYZ null null 1.00}}

\newenvironment{bullets} {\vspace{-9pt}\begin{itemize}\itemsep0pt} {\end{itemize}\vspace{-9pt}}
\newenvironment{bulletnums} {\vspace{-9pt}\begin{enumerate}\itemsep0pt} {\end{enumerate}\vspace{-9pt}}
\newcommand{\together}[1]{\Needspace*{#1\baselineskip}}

\newcommand{\AAA}{\mathcal{A}}
\DeclareMathOperator{\asc}{\mathsf{asc}}
\DeclareMathOperator{\av}{\mathsf{Av}}
\newcommand{\BBB}{\mathcal{B}}
\newcommand{\geqs}{\geqslant}
\newcommand{\leqs}{\leqslant}
\newcommand{\liminfty}[1][n]{\lim\limits_{#1\rightarrow\infty}}
\newcommand{\mxp}{M_\prec}
\newcommand{\nfrac}[2]{{}^{#1}\!/\!{}_{#2}}
\newcommand{\PPP}{\mathcal{P}}
\newcommand{\QQQ}{\mathcal{Q}}
\newcommand{\preq}{\preccurlyeq}
\newcommand{\veps}{\varepsilon}

\newcommand{\nl}{{\small NL}}
\newcommand{\nls}{{\footnotesize NL}}
\newcommand{\thf}{{\bf3}-free}
\newcommand{\thfnl}{\thf{} \nl{}}
\newcommand{\thttf}{\texorpdfstring{$\{${\bf3}, {\bf2}+{\bf2}$\}$}{\{3, 2+2\}}-free}
\newcommand{\thttfnl}{\thttf{} \nl{}}
\newcommand{\ttf}{({\bf2}+{\bf2})-free}
\newcommand{\ttfnl}{\ttf{} \nl{}}

\newcommand{\ii}{red}
\newcommand{\oo}{blue}
\newcommand{\io}{\ii{}-\oo{}}
\newcommand{\ipat}{\textsf{[}3--12}

\newcommand{\nlmx}{\big(\begin{smallmatrix}1&0\\{\bf1}&1\end{smallmatrix}\big)}
\newcommand{\thfmx}{\big(\begin{smallmatrix}1&\\{\bf1}&1\end{smallmatrix}\big)}

\newcommand{\pamx}{\Big(\begin{smallmatrix}1&0&0\\{\bf1}&0&0\\&{\bf1}&1\end{smallmatrix}\Big)}
\newcommand{\pbmx}{\Big(\begin{smallmatrix}0&1&0\\{\bf1}&0&1\\&{\bf1}&0\end{smallmatrix}\Big)}
\newcommand{\pcmx}{\Big(\begin{smallmatrix}0&0&1\\{\bf1}&1&0\\&{\bf1}&0\end{smallmatrix}\Big)}
\newcommand{\paaamx}{\begin{smallmatrix}i&{\bf\color{gray}1}&1&0&0\\j&&{\bf1}&0&0\\k&&&{\bf1}&1\\\ell&&&&{\bf\color{gray}1}\\[2pt]&i&j&k&\ell\end{smallmatrix}}
\newcommand{\pbbbmx}{\begin{smallmatrix}i&{\bf\color{gray}1}&0&1&0\\k&&{\bf1}&0&1\\j&&&{\bf1}&0\\\ell&&&&{\bf\color{gray}1}\\[2pt]&i&k&j&\ell\end{smallmatrix}}
\newcommand{\pcccmx}{\begin{smallmatrix}i&{\bf\color{gray}1}&0&0&1\\k&&{\bf1}&1&0\\\ell&&&{\bf1}&0\\j&&&&{\bf\color{gray}1}\\[2pt]&i&k&\ell&j\end{smallmatrix}}

\newcommand{\plotptradius}{0.275}
\newcommand{\setplotptradius}[1]{\renewcommand{\plotptradius}{#1}}

\newcommand{\plotpt}[3][] 
{ \fill[#1,radius=\plotptradius] (#2,#3) circle; }

\newcommand{\circpt}[3][] 
{ \draw[#1,radius=\plotptradius+0.075] (#2,#3) circle; }

\newcommand{\plotpermnobox}[3][]  
{
  \foreach \y [count=\x] in {#3}
  {
    \ifnum0=\y {} \else {
      \plotpt[#1]{\x}{\y}
    } \fi
  }
}

\newcommand{\plotperm}[3][]  
{
  \plotpermnobox[#1]{#2}{#3}
  \draw[thick] (.5,.5) rectangle (#2.5,#2.5);
}

\newtheorem{thm}{Theorem}
\newtheorem{prop}[thm]{Proposition}
\newtheorem{conj}[thm]{Conjecture}

\title{\textbf{On naturally labelled posets\\and permutations avoiding 12--34}}

\author{David Bevan${}^\dagger$, Gi-Sang Cheon${}^\ddag$ and Sergey Kitaev${}^\dagger$}

\date{}

\begin{document}
\maketitle

{\begin{NoHyper}
\let\thefootnote\relax\footnotetext
{${}^\dagger$Department of Mathematics and Statistics, University of Strathclyde, Glasgow, Scotland.}
\let\thefootnote\relax\footnotetext
{${}^\ddag$Department of Mathematics, Sungkyunkwan University, Suwon, South Korea.}
\end{NoHyper}}

{\begin{NoHyper}
\let\thefootnote\relax
\end{NoHyper}}

\begin{abstract}
\noindent
A partial order $\prec$ on $[n]$ is naturally labelled (\nls{}) if $x\prec y$ implies $x<y$.
We establish a bijection between \thttf{} \nls{} posets and 12--34-avoiding permutations,
determine functional equations satisfied by their generating function, and use series analysis to investigate their asymptotic growth,
presenting evidence of stretched exponential behaviour.
We also exhibit bijections between \thf{} \nls{} posets and various other objects, and determine their generating function.
The connection between our results and a hierarchy of combinatorial objects related to interval orders is described.

\end{abstract}

\section{Introduction}

A partial order $\prec$ on $[n]$ is said to be \emph{natural}~\cite{Avann1972} or \emph{naturally labelled}~\cite{CL2011,Stanley1972,StanleyEC1}
if $\prec$ is a suborder of the normal linear order on the integers.
That is, $\prec$ is naturally labelled if $x\prec y$ implies $x<y$.
For brevity, we often use \nl{} as an abbreviation for ``naturally labelled''.

The study of \nl{} posets goes back to Richard Stanley's PhD Thesis~\cite{Stanley1972}, and independently a few years later, to Kreweras~\cite{Kreweras1981}.
One nice application is Stanley's result~\cite{Stanley1973} that the evaluation of the chromatic polynomial of a graph at $-1$ gives the number of its acyclic orientations.
\nl{} posets also appear in the literature
in connection with the celebrated Neggers--Stanley conjecture on the real-rootedness of so-called $W$-polynomials~\cite{Branden2004,Neggers1978,Stembridge1997,Stembridge2007}.
For more on \nl{} posets, see Gessel's survey~\cite{Gessel2016}.
The counting sequence for \nl{} posets on $[n]$ is \href{http://oeis.org/A006455}{A006455} in the OEIS~\cite{OEIS},
and their asymptotic enumeration is given by Brightwell, Pr\"omel and Steger~\cite{BPS1996}.

Much research on posets has considered restricted classes, yielding many interesting results, perhaps most notably in~\cite{BMCDK2010} on \ttf{} posets and related structures.
Our focus in this work is on \nl{} posets that are \thf{} (having no 3-element chain), such as that on the left in Figure~\ref{fig3Free}.

In Section~\ref{sect3Free} we present bijections between \thfnl{} posets and a number of other equinumerous combinatorial objects before deducing their generating function.
Section~\ref{sect322Free} is the heart of the paper and concerns \thfnl{} posets which are also \ttf{} (having no induced subposet consisting of two disjoint 2-element chains).
Our main result is a bijection between \thttfnl{} posets and permutations avoiding the vincular pattern 12--34.
We also exhibit bijections with other equinumerous objects.
Definitions are given in the relevant sections.
The enumeration of these objects is also investigated, yielding functional equations satisfied by their generating function, a lower bound on
the exponential term in their asymptotics,
and (using series analysis) a conjecture that their number behaves asymptotically like
\[
A\cdot (\log4)^{-n} \cdot \mu^{n^{1/3}} \cdot n^\beta \cdot n! ,
\]
where estimates are given for the constants $A$, $\mu$ and $\beta$.

Finally, Section~\ref{sectDiscussion} places our new results in the context of a hierarchy of combinatorial objects related to interval orders.
This hierarchy, presented in Figures~\ref{figMainHierarchy} and~\ref{figNewHierarchy} on pages~\pageref{figMainHierarchy} and~\pageref{figNewHierarchy}, provides a frame of reference and additional motivation for this line of research.
At its heart is a collection of bijections between equinumerous combinatorial objects,
primarily classes of posets, families of matrices, pattern-avoiding permutations, and types of ascent sequence.
We extend the hierarchy by exhibiting bijections between objects of these sorts that are equinumerous to \thttfnl{} posets.
Avenues for future research are suggested.
For further details, see the discussion in Section~\ref{sectDiscussion}.

Given a partial order $\prec$ on $[n]$, its  \emph{incidence matrix} or \emph{poset matrix} is the $n\times n$ binary matrix $\mxp$ in which $\mxp(i,j)=1$ if and only if $i\preq j$.
There is a natural bijection between posets of~$[n]$ and their poset matrices.
The following proposition characterises the incidence matrices of \nl{} posets.
See~\cite{CCKMM2022} for a characterisation of which \nl{} poset matrices can be represented as binary \emph{Riordan} matrices.

\begin{prop}\label{propNLMatrices}
Naturally labelled posets on $[n]$ are in bijection with upper-triangular $n\times n$ binary matrices
with each entry on the main diagonal equal to one
that have no $\nlmx$ submatrix whose lower left entry (shown in bold) is on the main diagonal.
\end{prop}
\begin{proof}
  Reflexivity implies that
  each entry on the main diagonal of a poset matrix is equal to one.
  A partial order $\prec$ is \nl{} if and only if its poset matrix is upper-triangular since $i>j$ implies $\mxp(i,j)=0$.
  Being upper-triangular automatically entails antisymmetry: if $i\neq j$ then $\mxp(i,j)=1$ implies $\mxp(j,i)=0$.
  Finally, $\mxp$ contains $\nlmx$ as a submatrix with its lower left entry on the main diagonal if and only if there exist $i<j<k$ such that $\mxp(i,j)=1$ and $\mxp(j,k)=1$ but $\mxp(i,k)=0$, or equivalently, if $i\prec j$ and $j\prec k$, but $i\nprec k$, contradicting transitivity.
\end{proof}

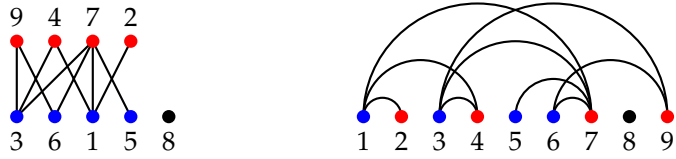
\begin{figure}[t]
  \centering
  \begin{tikzpicture}[scale=0.25]
  \setplotptradius{.35}
  \draw[thick] (1,1)--(5,5)--(7,1);
  \draw[thick] (5,1)--(3,5)--(1,1)--(1,5)--(3,1)--(5,5)--(5,1)--(7,5);
  \plotpermnobox[red]{}{5,0,5,0,5,0,5}
  \plotpermnobox[blue]  {}{1,0,1,0,1,0,1}
  \plotpermnobox        {}{0,0,0,0,0,0,0,0,1}
  \node at (1,-.3) {\small3};
  \node at (3,-.3) {\small6};
  \node at (5,-.3) {\small1};
  \node at (7,-.3) {\small5};
  \node at (9,-.3) {\small8};
  \node at (1,6.3) {\small9};
  \node at (3,6.3) {\small4};
  \node at (5,6.3) {\small7};
  \node at (7,6.3) {\small2};
  \end{tikzpicture}
  \hspace{.75in}
  \begin{tikzpicture}[scale=0.25]
  \setplotptradius{.35}
  \newcommand{\edge}[2]{\draw [thick] (#1,1) arc [radius=#2,start angle=0,end angle=180];}
  \edge{3}{1}
  \edge{7}{1}
  \edge{7}{3}
  \edge{13}{1}
  \edge{13}{2}
  \edge{13}{4}
  \edge{13}{6}
  \edge{17}{3}
  \edge{17}{6}
  \plotpermnobox[blue]  {}{1,0,0,0,1,0,0,0,1,0,1}
  \plotpermnobox[red]{}{0,0,1,0,0,0,1,0,0,0,0,0,1,0,0,0,1}
  \plotpermnobox        {}{0,0,0,0,0,0,0,0,0,0,0,0,0,0,1}
  \node at ( 1,-.3) {\small1};
  \node at ( 3,-.3) {\small2};
  \node at ( 5,-.3) {\small3};
  \node at ( 7,-.3) {\small4};
  \node at ( 9,-.3) {\small5};
  \node at (11,-.3) {\small6};
  \node at (13,-.3) {\small7};
  \node at (15,-.3) {\small8};
  \node at (17,-.3) {\small9};
  \end{tikzpicture}
  \caption{The Hasse diagram of a \thfnl{} poset, and the corresponding Stanley graph}\label{fig3Free}
\end{figure}

\section{\thf{} naturally labelled posets}\label{sect3Free}

A poset is \emph{\thf} if it has no 3-element chains.
That is, $\prec$ is \thf{} if there are no three elements $x\prec y\prec z$.
Every element of a \thf{} poset is thus either minimal or maximal (with isolated elements being both minimal and maximal).
See Figure~\ref{fig3Free} for an example.

In this section, we consider \thfnl{} posets. We exhibit bijections between these posets and certain matrices and certain graphs, before determining the generating function for these objects.
We conclude by establishing the generating function for those \thfnl{} posets without any isolated elements.

The following proposition characterises the incidence matrices of \thfnl{} posets.

\begin{prop}\label{prop3FreeMatrices}
The \thf{} naturally labelled posets on $[n]$ are in bijection with upper-triangular \mbox{$n\times n$} binary matrices
with each entry on the main diagonal equal to one
that contain no $\thfmx$ partial submatrix whose lower left (bold) entry is on the main diagonal.
\end{prop}
\begin{proof}
  By Proposition~\ref{propNLMatrices}, the poset matrix of a \nl{} partial order $\prec$ is upper-triangular with
  each entry on the main diagonal equal to one.
  The matrix $\mxp$ contains $\thfmx$ with its lower left entry on the main diagonal
  if and only if there exist $i<j<k$ such that $\mxp(i,j)=1$ and $\mxp(j,k)=1$, or equivalently, if $i\prec j\prec k$.
  Thus the poset matrix of a \nl{} partial order is $\thfmx$-free if and only if the poset is \thf.
\end{proof}

A labelled graph with vertex set $[n]$ is a \emph{Stanley graph}~\cite{Stanley1997,SL1999} if and only if no vertex $v$ has both left neighbours \mbox{($u<v$)} and right neighbours ($u>v$).
Thus every Stanley graph is bipartite: every edge connects a vertex with (only) right neighbours to a vertex with (only) left neighbours.
There is a natural bijection between \thfnl{} posets and Stanley graphs since Stanley graphs are exactly the Hasse diagrams of \thfnl{} posets.
See Figure~\ref{fig3Free} for an example.

\begin{prop}\label{prop3FreeStanley}
The \thf{} naturally labelled posets on $[n]$ are in bijection with $n$-vertex Stanley graphs.
\end{prop}

The following proposition presents enumerative results concerning \thfnl{} posets, including a recurrence relation showing a relationship with a $q$-analogue of the Stirling numbers of the second kind, and both a functional equation and an explicit expression for the generating function.
The counting sequence for \thfnl{} posets on $[n]$ is \href{http://oeis.org/A135922}{A135922} in~\cite{OEIS}.

\begin{prop}\label{prop3FreeEnum}
Let $p(n,k)$ be the number of \thf{} naturally labelled posets on $[n]$ with $k$ minimal elements. Then
\begin{equation}\label{eq3FreeRecurrence}
  p(n,k) \;=\; p(n-1,k-1) \:+\: (2^k -1)p(n-1,k)
\end{equation}
where $p(0,0)=1$, $p(n,0)=0$ if $n\geqs1$, and $p(n,k)=0$ if $n<k$.
Thus, $p(n,k)=S_2[n,k]$, where the $S_q[n,k]$ are the $q$-Stirling numbers of the second kind.

Suppose
\begin{equation}\label{eq3FreeGFDef}
F(z,y)\;=\;\sum_{n\geqs k\geqs0}p(n,k)z^ny^k
\end{equation}
is the corresponding bivariate generating function. 
Then $F$ satisfies the functional equation
\begin{equation}\label{eq3FreeFuncEq}
F(z,y) \;=\; 1 \:+\: z \big( F(z,2y) - (1-y)F(z,y) \big) ,
\end{equation}
and can be expressed explicitly by
\begin{equation}\label{eq3FreeGF}
F(z,y) \;=\; \sum_{k\geqs0}\frac{z^ky^k}{\prod_{i=1}^k \big(1-(2^i-1)z\big)} .
\end{equation}
\end{prop}
\begin{proof}
  The unique empty poset gives $p(0,0)=1$.
  Every nonempty poset has at least one minimum, so $p(n,0)=0$ if $n\geqs1$, and no poset has more minima than elements, so $p(n,k)=0$ if $n<k$.

  The nonempty \thfnl{} posets on $[n]$ with $k$ minima are of two types: (i)~those in which $n$ is minimal, and (ii)~those in which $n$ is not minimal.
  In type~(i), $n$ is isolated.
  Thus, if we remove~$n$ we obtain a \thfnl{} poset on $[n-1]$ with $k-1$ minima.
  So there are $p(n-1,k-1)$ posets of type~(i).

  Now consider a poset of type~(ii).
  Since $n$ is not minimal, if we remove $n$ we obtain a \thfnl{} poset on $[n-1]$ with $k$ minima.
  For every such poset, a non-minimal $n$ can be added to cover any nonempty subset of the $k$ minima, of which there are $2^k-1$.
  So there are $(2^k -1)p(n-1,k)$ posets of type~(ii).
  Hence $p(n,k)$ satisfies recurrence~\eqref{eq3FreeRecurrence}.
  This is identical to the recurrence of Carlitz for $q$-Stirling numbers of the second kind with $q=2$ (see~\cite{CER2018}):
  \[
  S_q[n,k] \;=\; S_q[n-1,k-1] \:+\: [k]_q \cdot S_q[n-1,k],
  \]
  where $[k]_q = 1 + q + \ldots + q^{k-1}$. Thus $p(n,k)=S_2[n,k]$ (see~\href{https://oeis.org/A139382}{A139382} in~\cite{OEIS}).

  If we reverse our decomposition, and consider the process of adding a new maximal element to a \thfnl{} poset
  with $k$ minima, then the new element may cover any subset of the $k$ minimal elements, so there are $2^k$ ways of adding a new element. Exactly one of these (covering no minima) increases the number of minima to $k+1$; all the others leave $k$ minimal elements.

  Now, in $F(z,y)$, each monomial $z^ny^k$ represents a \thfnl{} poset on $[n]$ with $k$ minima.
  So the process of adding a new maximal element is represented by the following operation on monomials:
  \[
  z^ny^k \;\mapsto\; z^{n+1} \big( y^{k+1}+(2^k-1)y^k \big) \;=\; z^{n+1} \big( (2y)^k - (1-y)y^k \big) .
  \]
  Thus, by extending to a linear operator on the generating function and including the empty poset as the base case, we see that $F(z,y)$ satisfies the functional equation~\eqref{eq3FreeFuncEq}.

  If we let $P_k(z)=\sum_{n\geqs0}p(n,k)z^n$, then, for each $k\geqs1$, the recurrence relation~\eqref{eq3FreeRecurrence} gives
  \[
  P_k(z) \;=\; z \big( P_{k-1}(z)+(2^k -1)P_k(z) \big) .
  \]
  Thus,
  \[
  P_k(z) \;=\; \frac{zP_{k-1}(z)}{1-(2^k-1)z}.
  \]
  Iterating, with $P_0(z)=1$, gives, for each $k\geqs0$,
  \[
  P_k(z) \;=\; \frac{z^k}{(1-z)(1-3z)\ldots(1-(2^k-1)z)} \;=\; \frac{z^k}{\prod_{i=1}^k \big(1-(2^i-1)z\big)}.
  \]
  The identity $F(z,y)=\sum_{k\geqs0}P_k(z)y^k$ then yields the explicit expression~\eqref{eq3FreeGF} for $F(z,y)$.
\end{proof}

The asymptotic number of \thfnl{} posets on $[n]$ is given in an unpublished comment by Vaclav Kot\v{e}\v{s}ovec as $c\cdot 2^{n^2/4}$, where $c \approx 7.37196880$ if $n$ is even, and $c \approx 7.37194949$ if $n$ is odd; see \href{http://oeis.org/A135922}{A135922} in~\cite{OEIS}.

\together7
We conclude this section by presenting a new explicit formula for the enumeration of those \thfnl{} posets which have no isolated elements (see~\href{https://oeis.org/A323842}{A323842} in~\cite{OEIS}).

\begin{prop}\label{prop3FreeNoIsolated}
Suppose $G(z,y)$ is the bivariate generating function in which the coefficient of $z^n y^k$ is the number of \thf{} naturally labelled posets on $[n]$ with no isolated elements and $k$ minima.
Then $G$ satisfies the functional equation
\begin{equation}\label{eq3FreeNoIsolatedFuncEq}
G(z,y) \;=\;
\frac{1}{1-z^2 y^2} \left(1 - z y + z \left( G\!\left(\!\frac{z}{1-z y},\,2 y\!\right)-(1-y) (1-z y) G(z,y)\right)\right)
\end{equation}
and can be expressed explicitly by
\begin{equation}\label{eq3FreeNoIsolatedGF}
G(z,y) 
\;=\; \frac1{1+zy} \, \sum_{k\geqs0} \frac{z^ky^k}{\prod_{i=1}^k \big(1-(2^i-1-y)z\big)} .
\end{equation}
\end{prop}

\begin{figure}[t]
  \centering
  \begin{tikzpicture}[scale=0.25]
  \setplotptradius{.35}
  \draw[thick] (1,1)--(1,5)--(3,1)--(3,5);
  \draw[thick] (3,1)--(5,5);
  \plotpermnobox[red]{}{5,0,5,0,5}
  \plotpermnobox[blue]  {}{1,0,1}
  \plotpermnobox        {}{0,0,0,0,1,0,1,0,1,0,1}
  \node at (1,-.3) {\small2};
  \node at (3,-.3) {\small5};
  \node at (5,-.3) {\small1};
  \node at (7,-.3) {\small3};
  \node at (9,-.3) {\small4};
  \node at (11,-.3) {\small8};
  \node at (1,6.3) {\small9};
  \node at (3,6.3) {\small6};
  \node at (5,6.3) {\small7};
  \end{tikzpicture}
  \hspace{.75in}
  \begin{tikzpicture}[scale=0.25]
  \setplotptradius{.35}
  \draw[thick] (4,1)--(4,5)--(8,1)--(8,5);
  \draw[thick] (8,1)--(12,5);
  \plotpermnobox[red]{}{0,0,0,5,0,0,0,5,0,0,0,5}
  \plotpermnobox[blue]  {}{0,0,0,1,0,0,0,1}
  \setplotptradius{0.2}
  \circpt{1}{1}
  \circpt{5}{1}
  \circpt{6}{1}
  \circpt{13}{5}
  \node at (0,-.3) {\small0};
  \node at (4,-.3) {\small1};
  \node at (8,-.3) {\small2};
  \node at (4,6.3) {\small5};
  \node at (8,6.3) {\small3};
  \node at (12,6.3) {\small4};
  \end{tikzpicture}
  \caption{A \thfnl{} poset with four isolated vertices, and the corresponding decorated \thfnl{} poset with no isolated vertices}\label{fig3FreeNoIsolated}
\end{figure}
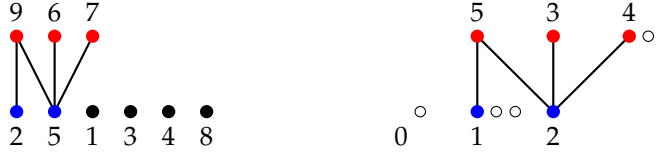

\begin{proof}
Let us say that a \emph{decorated poset} is a \thfnl{} poset with no isolated elements in which each element, including an additional invisible element~0, is decorated with a sequence of zero or more rings.
It is easy to see that \thfnl{} posets on $[n]$ with $\ell$ isolated elements are in bijection with decorated posets on $[n-\ell]$ with $\ell$ rings.
Given a \thfnl{} poset~$\PPP$, to construct the corresponding decorated poset~$\QQQ$, we simply process its elements in order.
If an element is isolated in~$\PPP$, add a ring to the previous element in~$\QQQ$; otherwise add the element to $\QQQ$ with the covering relations corresponding to those in~$\PPP$.
See Figure~\ref{fig3FreeNoIsolated} for an example.

In terms of generating functions, this means we have the following relationship between $G$ and $F$, the bivariate generating function~\eqref{eq3FreeGFDef} for all \thfnl{} posets:
\[
F(z,y) \;=\; \frac1{1-zy} \, G\!\left(\!\frac{z}{1-zy},\,y\!\right).
\]
Here, ${z}/(1-zy)$ accounts for decorated elements, and $1/(1-zy)$ for the sequence of rings on invisible element 0.
A simple change of variables then gives us $G$ in terms of~$F$:
\[
G(z,y) \;=\; \frac1{1+zy} \, F\!\left(\!\frac{z}{1+zy},\,y\!\right) .
\]
Substitution into the functional equation~\eqref{eq3FreeFuncEq} for $F$ and some algebraic rearrangement yields the functional equation for $G$~\eqref{eq3FreeNoIsolatedFuncEq}, and substitution into~\eqref{eq3FreeGF} gives the explicit formula~\eqref{eq3FreeNoIsolatedGF}.
\end{proof}

\section{\thttf{} naturally labelled posets}\label{sect322Free}

A poset is said to be \ttf{} if it does not contain an \emph{induced} subposet that is isomorphic to {\bf2}+{\bf2}, the union of two disjoint 2-element chains.
For example, in Figure~\ref{fig22Free}, the poset on the left is \ttf{}, but the poset on the right is not.
See Section~\ref{sectDiscussion} below for further discussion of \ttf{} posets.
\begin{figure}[ht]
  \centering
  \begin{tikzpicture}[scale=0.3]
  \draw[white] (1,2) circle [x radius=.8,y radius=1.7];
  \draw[thick] (1,3)--(3,1)--(3,5)--(1,1)--(1,5)--(3,3);
  \plotpermnobox{}{5,0,5}
  \plotpermnobox{}{3,0,3}
  \plotpermnobox{}{1,0,1}
  \end{tikzpicture}
  \hspace{.75in}
  \begin{tikzpicture}[scale=0.3]
  \draw[thick] (1,3)--(3,1)--(3,5);
  \draw[thick] (1,1)--(1,5)--(3,3);
  \draw[dashed] (1,2) circle [x radius=.8,y radius=1.7];
  \draw[dashed] (3,4) circle [x radius=.8,y radius=1.7];
  \plotpermnobox{}{5,0,5}
  \plotpermnobox{}{3,0,3}
  \plotpermnobox{}{1,0,1}
  \end{tikzpicture}
  \caption{A \ttf{} poset and a poset with an induced {\bf2}+{\bf2} subposet}\label{fig22Free}
\end{figure}
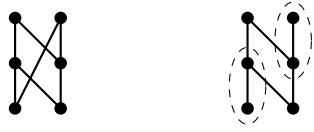

In this section we consider \thttfnl{} posets.
First, we exhibit bijections between these posets and certain matrices, certain labelled binary words, and certain bicoloured permutations.
Then we establish a bijection with permutations avoiding the vincular pattern 12--34.
We then consider the enumeration of these objects, first determining functional equations satisfied by their generating function, and then investigating their asymptotic growth.

The following proposition characterises the incidence matrices of \thttfnl{} posets.

\begin{prop}\label{prop322FreeMatrices}
The \thttf{} naturally labelled posets on $[n]$ are in bijection with upper-triangular $n\times n$ binary matrices with
each entry on the main diagonal equal to one
that contain none of the four partial submatrices
\[
M_0=\thfmx, \qquad
M_1=\pamx, \qquad 
M_2=\pbmx, \qquad 
M_3=\pcmx, 
\]
with lower left (bold) entries on the main diagonal.
\end{prop}
\begin{proof}
By Proposition~\ref{prop3FreeMatrices}, the incidence matrix of a \thfnl{} poset is upper-triangular with each entry on the main diagonal equal to one and has no $M_0$ partial submatrix.
Suppose a \nl{} partial order $\prec$ has an induced {\bf2}+{\bf2} subposet $i\prec j$, $k\prec\ell$ (where $i$ and $j$ are both incomparable with both $k$ and $\ell$ under~$\prec$).
Without loss of generality, we may assume that $i<k$.
Then three cases are possible:
\together3
\[
i<j<k<\ell, \qquad\qquad i<k<j<\ell, \qquad\qquad i<k<\ell<j.
\]
\[
\paaamx \qquad\qquad\qquad\quad \pbbbmx \qquad\qquad\qquad\quad \pcccmx
\]
These clearly correspond precisely to an occurrence in $\mxp$ of the partial submatrices $M_1$, $M_2$, or $M_3$, respectively, as illustrated.
Thus, if a naturally labelled poset is not \ttf{}, then its incidence matrix contains one of these three partial submatrices,
and if the incidence matrix contains any one of the partial submatrices, then the corresponding poset is not \ttf{}.
\end{proof}

It is well-known (see~\cite{BMCDK2010}) that a poset is \ttf{} if and only if the set of its strict downsets can be linearly ordered by inclusion.
That is, given a poset $(P,\prec)$, if for each $x\in P$, we let $D(x)=\{t\in P : t\prec x\}$ denote the strict downset of $x$, then for any pair of elements $x,y\in P$, either $D(x)\subseteq D(y)$ or $D(y)\subseteq D(x)$.
We use this characterisation to exhibit natural bijections between \thttfnl{} posets and certain sets of labelled binary words and bicoloured permutations.

\begin{prop}\label{prop322FreeWords}
  The \thttf{} naturally labelled posets on $[n]$ are in bijection with words over $\{\mathsf0,\mathsf1\}$ of length~$n$, with the letters labelled $\mathsf1,\ldots,\mathsf{n}$ and satisfying the following four conditions:
\begin{bulletnums}
  \item the first letter is $\mathsf0$,
\item the labels on adjacent $\mathsf0$s are in {decreasing} order,
  \item the labels on adjacent $\mathsf1$s are in {increasing} order, and
  \item the label on any $\mathsf1$ is greater than the labels on all the $\mathsf0$s earlier in the word.
\end{bulletnums}
\end{prop}

\begin{proof}
  Given a \thttfnl{} poset $\PPP$, let $D_1\subsetneq D_2\subsetneq \ldots \subsetneq D_k$ be the 
  {distinct} strict downsets of its non-isolated maxima.
  Each $D_i$ is a nonempty subset of the minima of $\PPP$.
  Let $E_1=D_1$, and for $i=2,\ldots,k$, let $E_i=D_i\setminus D_{i-1}$.
  Finally, let $E_{k+1}$ be the set of isolated elements of $\PPP$.
  Then the $E_i$'s form a partition of the minima of $\PPP$, with each part nonempty, except possibly~$E_{k+1}$.

  For each $i\in[k]$, let $M_i=\{x\in\PPP: D(x)=D_i\}$ be the set of maxima of $\PPP$ whose strict downset is equal to $D_i$.
  The $M_i$'s partition the non-isolated maxima of $\PPP$, and each $M_i$ is nonempty.

  Clearly $\PPP$ is defined by the sequence $(E_1,M_1,E_2,M_2,\ldots,E_k,M_k,E_{k+1})$.
  We encode this as a labelled binary word $\veps_1\mu_1\ldots\veps_k\mu_k\veps_{k+1}$, in which
  each $\veps_i$ consists of $|E_i|$ $\mathsf0$s, labelled with the elements of $E_i$ in decreasing order, and
  each $\mu_i$ consists of $|M_i|$ $\mathsf1$s, labelled with the elements of $M_i$ in increasing order\footnote{The choice of decreasing labels for minima and increasing labels for maxima is motivated by the requirements of our bijection between \thttf{} NL posets and permutations avoiding 43--12 in Proposition~\ref{propAv1234Bijection}.}, so that each subset of elements is uniquely represented and the sets comprising the sequence can be distinguished.
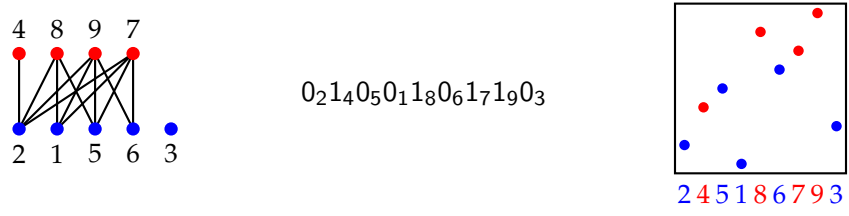
\begin{figure}[t]
  \centering
  \raisebox{15pt}{
  \begin{tikzpicture}[scale=0.25]
  \setplotptradius{.35}
  \draw[thick] (1,5)--(1,1);
  \draw[thick] (3,5)--(1,1) (3,5)--(3,1) (3,5)--(5,1);
  \draw[thick] (5,5)--(1,1) (5,5)--(3,1) (5,5)--(5,1) (5,5)--(7,1);
  \draw[thick] (7,5)--(1,1) (7,5)--(3,1) (7,5)--(5,1) (7,5)--(7,1);
  \plotpermnobox[red]{}{5,0,5,0,5,0,5}
  \plotpermnobox[blue]  {}{1,0,1,0,1,0,1,0,1}
  \node at (1,-.3) {\small2};
  \node at (3,-.3) {\small1};
  \node at (5,-.3) {\small5};
  \node at (7,-.3) {\small6};
  \node at (9,-.3) {\small3};
  \node at (1,6.3) {\small4};
  \node at (3,6.3) {\small8};
  \node at (5,6.3) {\small9};
  \node at (7,6.3) {\small7};
  \end{tikzpicture}
  \hspace{0.5in}
  \raisebox{27pt}{$\mathsf{0_2 1_4 0_5 0_1 1_8 0_6 1_7 1_9 0_3}$}
  }
  \hspace{0.5in}
  \begin{tikzpicture}[scale=0.25]
  \plotperm[red]{9}  {0,4,0,0,8,0,7,9}
  \plotpermnobox[blue]{}{2,0,5,1,0,6,0,0,3}
  \small
  \node at (2,-.6) {\color{red}4};
  \node at (5,-.6) {\color{red}8};
  \node at (7,-.6) {\color{red}7};
  \node at (8,-.6) {\color{red}9};
  \node at (1,-.6) {\color{blue}2};
  \node at (3,-.6) {\color{blue}5};
  \node at (4,-.6) {\color{blue}1};
  \node at (6,-.6) {\color{blue}6};
  \node at (9,-.6) {\color{blue}3};
  \end{tikzpicture}
  \caption{The Hasse diagram of a \thttfnl{} poset, and the corresponding labelled binary word and bicoloured permutation}\label{fig322Free}
\end{figure}

\together7
For example, for the poset on the left in Figure~\ref{fig322Free}, we have
  \begin{center}
    \begin{tabular}{llll}
      $E_1=\{2\}$,   & $\veps_1=\mathsf{0_2}$,    & $M_1=\{4\}$,   & $\mu_1=\mathsf{1_4}$,      \\[2pt]
      $E_2=\{1,5\}$, & $\veps_2=\mathsf{0_50_1}$, & $M_2=\{8\}$,   & $\mu_2=\mathsf{1_8}$,      \\[2pt]
      $E_3=\{6\}$,   & $\veps_3=\mathsf{0_6}$,    & $M_3=\{7,9\}$, & $\mu_3=\mathsf{1_71_9}$,   \\[2pt]
      $E_4=\{3\}$,   & $\veps_4=\mathsf{0_3}$,
    \end{tabular}
  \end{center}
  yielding the word in the centre of the figure.

  Any word created in this way satisfies the four conditions in the statement of the proposition:
  Condition 1 holds because $E_1$ is nonempty,
  conditions 2 and 3 hold as a result of our choices for the ordering of labels,
  and condition 4 holds because $\PPP$ is naturally labelled.

  Moreover, a unique poset can be built from any word satisfying the four conditions.
  Reading from left to right, $\textsf0_p$ adds a new (initially isolated) minimal element $p$,
  whereas $\textsf1_q$ adds a new maximal element $q$ covering all existing minima.
\end{proof}

A reinterpretation of these words yields a bijection with certain bicoloured permutations.
An \emph{inversion} in a permutation consists of two values (not necessarily consecutive), the larger of which occurs first.
A \emph{descent} is an inversion consisting of two consecutive values.
An \emph{ascent} consists of two consecutive values, the smaller of which occurs first.

\begin{prop}\label{prop322FreePerms}
  The \thttf{} naturally labelled posets on $[n]$ are in bijection with permutations of length~$n$, each entry of which is coloured \oo{} or \ii, whose
  first entry is~\oo{}, and which avoid \oo-\oo{} ascents, \ii-\ii{} descents, and \oo{}--\ii{} inversions.
\end{prop}
\begin{proof}
  These bicoloured permutations are simply an alternative representation of the labelled binary words of Proposition~\ref{prop322FreeWords},
  formed by switching the role of letters and labels and using the colours \oo{} and \ii{} for $\mathsf0$ and  $\mathsf1$, respectively.
  See Figure~\ref{fig322Free} for an example.
  The four conditions on the colouring of entries precisely correspond to the four conditions in Proposition~\ref{prop322FreeWords} on the labelled binary words.
\end{proof}

\together5
\subsection{Bijection with permutations avoiding 12--34}

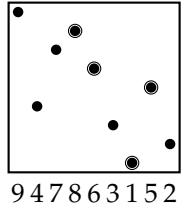
\begin{figure}[t]
  \centering
  \begin{tikzpicture}[scale=0.25]
  \plotperm{9}{9,4,7,8,6,3,1,5,2}
  \circpt48
  \circpt56
  \circpt71
  \circpt85
  \small
  \node at (1,-.6) {9};
  \node at (2,-.6) {4};
  \node at (3,-.6) {7};
  \node at (4,-.6) {8};
  \node at (5,-.6) {6};
  \node at (6,-.6) {3};
  \node at (7,-.6) {1};
  \node at (8,-.6) {5};
  \node at (9,-.6) {2};
  \end{tikzpicture}
  \caption{The single occurrence of the pattern 43--12 in the permutation 947863152}\label{figVincular}
\end{figure}

A \emph{vincular pattern} is a permutation together with adjacency conditions for its containment.
We use the traditional notation in which a vincular pattern is written as a permutation with dashes inserted between
terms that need not be adjacent but no dashes between terms that must be adjacent.
For example, the permutation $\sigma$ contains the vincular pattern 43--12 if there exist indices $i<j$ such that
$\sigma(i) > \sigma(i+1) > \sigma(j+1) > \sigma(j)$.
The permutation in Figure~\ref{figVincular} contains a single occurrence of 43--12.
Vincular patterns were introduced in~\cite{BS2000} (under another name) and first called vincular patterns in~\cite{BMCDK2010}.

We also require the ability to specify a pattern in which the first term must occur at the start of the host permutation.
We denote this with an initial bracket.
For example, the permutation~$\sigma$ contains the pattern \ipat{} if there exists an index $i$ such that
$\sigma(1) > \sigma(i+1) > \sigma(i)$.
The permutation in Figure~\ref{figVincular} contains two occurrences of \ipat{}, formed by 947 and 915.

\together5
In this section we prove that \thttfnl{} posets have the same counting sequence as permutations avoiding the vincular pattern 12--34.
The permutations avoiding 12--34 were first studied by Elizalde~\cite[Section~5]{Elizalde2006}, who, among other things,
established that 12--34-avoiders are equinumerous with both 12--43-avoiders and 21--43-avoiders (and their symmetries).\footnote{That is, 12--34, 12--43 and 21--43 are \emph{Wilf equivalent}.}

\begin{prop}[{see~\cite[Propositions 5.2 and 5.3]{Elizalde2006}}]\label{propAv1234Sym}
The class $\av(\text{\textup{12--34}})$ of permutations avoiding the vincular pattern \textup{12--34} is equinumerous with both $\av(\text{\textup{12--43}})$ and $\av(\text{\textup{21--43}})$, and also with the five other symmetries of these three classes.
\end{prop}

In light of this, the heart of our proof consists of the establishment of a bijection between the bicoloured permutations of $[n]$ defined in Proposition~\ref{prop322FreePerms} and permutations of $[n]$ avoiding 43--12.

Let $\BBB$ denote the set of bicoloured permutations defined in Proposition~\ref{prop322FreePerms}.
Observe that the runs (maximal sequences of ascending or descending points) of any $\beta\in\BBB$ are coloured as follows:
\begin{bullets}
  \item The first point in an \emph{ascending} run of $\beta$ is \oo{} (since the first point of $\beta$ is \oo{}, and \ii-\ii{} descents and \oo{}--\ii{} inversions are forbidden). Subsequent points may have either colour, but a pair of consecutive points cannot both be \oo{} (since \oo-\oo{} ascents are forbidden). In particular, the second point of an ascending run is \ii{}.
  \item The first point in a \emph{descending} run of $\beta$ may have either colour, but subsequent points are all \oo{} (avoiding \ii-\ii{} descents and \oo{}--\ii{} inversions).
\end{bullets}

We define $\Lambda$ to be the following length-preserving map from $\BBB$ to uncoloured permutations:
\begin{bullets}
  \item If $\beta$ is a bicoloured permutation that doesn't contain a \io{} ascent, then $\Lambda(\beta)$ is simply the permutation that results from removing the colours from~$\beta$.
  \item\label{LambdaB1} If $\beta$ contains one or more \io{} ascents, then for each such ascent $PQ$, we \emph{mark} the point $Q$ and any points in a descending run beginning with $Q$ that lie above the \ii{} point~$P$. All these points are \oo{}.
      In this case, $\Lambda(\beta)$ is the uncoloured permutation that results from moving all the marked points to the start of the permutation, placing them in increasing order.
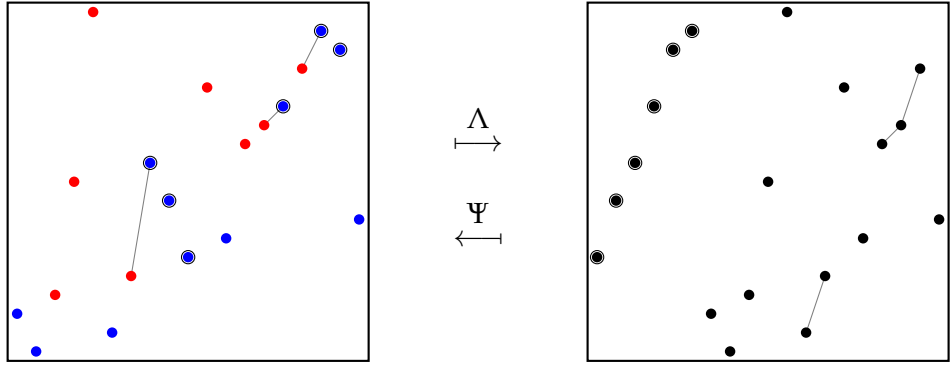
\begin{figure}[t]
\begin{center}
  \begin{tikzpicture}[scale=0.25]
    \draw[gray] (7,5)--(8,11);
    \draw[gray] (14,13)--(15,14);
    \draw[gray] (16,16)--(17,18);
    \plotperm[blue]       {19}{3,1,0,00,00,2,0,11,9,6,00,7,00,00,14,00,18,17,8}
    \plotpermnobox[red]{19}{0,0,4,10,19,0,5,00,0,0,15,0,12,13,00,16,00,00,0}
    \circpt{8}{11}
    \circpt{9}{9}
    \circpt{10}{6}
    \circpt{15}{14}
    \circpt{17}{18}
    \circpt{18}{17}
  \end{tikzpicture}
  \raisebox{66pt}{$\qquad\begin{tabular}{c} $\Lambda$ \\[-5pt] $\longmapsto$ \\ ~ \\ $\Psi$ \\[-5pt] $\longmapsfrom$ \end{tabular}
  \qquad$}
  \begin{tikzpicture}[scale=0.25]
    \draw[gray] (16,12)--(17,13)--(18,16);
    \draw[gray] (12,2)--(13,5);
    \plotperm       {19}{6,9,11,14,17,18,3,1,0,00,00,2,0,00,7,00,00,00,8}
    \plotpermnobox  {19}{0,0,00,00,00,00,0,0,4,10,19,0,5,15,0,12,13,16,0}
    \circpt{3}{11}
    \circpt{2}{9}
    \circpt{1}{6}
    \circpt{4}{14}
    \circpt{6}{18}
    \circpt{5}{17}
  \end{tikzpicture}
\end{center}
\caption{A bicoloured permutation
in $\BBB$,
and the 43--12 avoider that corresponds to it under the bijection $\Lambda$}
\label{figLambda}
\end{figure}
See Figure~\ref{figLambda} for an example; on the left, \io{} ascents are indicated with a grey line, and the marked points are circled.
\end{bullets}

We claim that the map $\Lambda$ is a bijection from $\BBB$ to $\av(\text{\textup{43--12}})$.

  We begin with some further properties of elements of $\BBB$.
  Suppose $\beta\in\BBB$.
  Then $\beta$ avoids 43--12, because the lower point in each descent of $\beta$ is coloured \oo{} and at least one of the points in each ascent of $\beta$ is \ii{}, so any 43--12 would contain a forbidden \oo{}--\ii{} inversion, realized by the 3 and one of the points of the 12.
  Moreover, $\beta$ also avoids the pattern \ipat{}, since the first point of $\beta$ is \oo{},
  so any \ipat{} would contain a forbidden \oo{}--\ii{} inversion, again realized by the 3 and one of the points of the 12.
  Thus, the removal of the colours from $\beta$ yields a permutation in $\AAA_0 := \av\!\big(\text{43--12}, \text{\ipat{}}\big)$.

  Let $\BBB_0$ consist of those elements of $\BBB$ that don't contain a \io{} ascent.
  Recall that $\Lambda$ simply removes the colours from bicoloured permutations in $\BBB_0$.
  We now show that $\Lambda$ maps $\BBB_0$ bijectively to~$\AAA_0$. 

  Suppose $\beta\in\BBB_0$, so $\beta$ contains no \io{} ascent.
  Then the points must be coloured as follows:
  The first point of $\beta$ is blue.
  In an ascending run of $\beta$, the first point is blue, and all subsequent points are red.
  In a descending run of $\beta$, the first point is either the first point of $\beta$ (coloured blue) or else the last point of an ascending run (coloured red),
  and all subsequent points are blue.
  No alternative colouring is possible.
  Thus, $\Lambda$ maps $\BBB_0$ injectively into~$\AAA_0$. 

  On the other hand, suppose $\sigma\in\AAA_0$. 
  Then the points of $\sigma$ can be coloured to produce an element of $\BBB_0$, 
  by colouring the first point and the non-initial points of descending runs \oo{}, and all the other points (the non-initial points of ascending runs) \ii{}.
  This clearly avoids \oo-\oo{} ascents, \io{} ascents and \ii-\ii{} descents.
  And if the result contained a \oo{}--\ii{} inversion $P$--$Q$, then $P$ would be either the first point of $\sigma$ or the second point of a descent and $Q$ would be the second point in an ascent, so $\sigma$ would contain either a \ipat{} or a 43--12.
  So, the restriction of $\Lambda$ to $\BBB_0$ is a bijection between $\BBB_0$ and~$\AAA_0$. 

  We now consider the effect of $\Lambda$ on bicoloured permutations that contain at least one \io{} ascent.
  Let $\BBB_1=\BBB\setminus\BBB_0$ be the set consisting of such bicoloured permutations.
  Also, let $\AAA_1 = \av(\text{43--12})\setminus\AAA_0$ be comprised of those 43--12 avoiders that contain at least one occurrence of \ipat{}.
  We need to show that $\Lambda$ maps $\BBB_1$ bijectively to~$\AAA_1$.

  Recall 
  the effect of $\Lambda$ on elements of $\BBB_1$:
  For each \io{} ascent $PQ$,
  the upper point~$Q$ is marked along with any points in a descending run beginning with $Q$ that lie above $P$.
  Then all the marked points are moved to the start of the permutation, placed in increasing order.

  Suppose that $\beta\in\BBB_1$, so $\beta$ contains a \io{} ascent.
  We first establish that $\Lambda(\beta)$ avoids 43--12.
  Recall that $\beta$ avoids both 43--12 and \ipat{}.

  The marked points of $\Lambda(\beta)$ form its initial ascending run (which may be a single point), since the first point of $\beta$ is below the lowermost marked point --- which becomes the first point of~$\Lambda(\beta)$; otherwise there would be a \oo{}--\ii{} inversion in $\beta$, formed of its first point and the lower point of the leftmost \io{} ascent.

  Thus, $\Lambda(\beta)$ avoids 43--12:
\begin{bullets}
      \item[(a)] If the descent created in $\Lambda(\beta)$ by the uppermost marked point and the first point of $\beta$ were to form the 43 of a 43--12, then $\beta$ would have contained a \ipat{}.
    \item[(b)] If the deletion of a descending run of marked points were to create a descent that formed the 43 of a 43--12 in $\Lambda(\beta)$, then the lowest of these marked points would have been the first point of a 43--12 in~$\beta$.
    \item[(c)] If the deletion of a descending run of marked points were to create an ascent that formed the 12 of a 43--12 in $\Lambda(\beta)$, then the lowest of these marked points would have been the third point of a 43--12 in~$\beta$.
\end{bullets}
We now establish that $\Lambda(\beta)$ contains \ipat{}.

The points in all the \io{} ascents of $\beta$ form an increasing subsequence; otherwise $\beta$ would contain a \oo{}--\ii{} inversion.
Each \io{} ascent of $\beta$ consists of the upper two points of a \oo-\ii-\oo{} or \ii-\ii-\oo{} double ascent, since a point coloured \ii{} can't be the first point in an ascending run.
Thus, $\Lambda(\beta)$ contains a \ipat{}, formed from the lowest marked point and the lower two points of the leftmost \oo-\ii-\oo{} or \ii-\ii-\oo{} double ascent.
Hence, $\Lambda$~maps $\BBB_1$ to~$\AAA_1$.

Our goal now is to construct a map from $\AAA_1$ to $\BBB_1$ and to prove that it is in fact the partial inverse of~$\Lambda$.

  Suppose $\sigma\in\AAA_1$.
  For each point $P$ in the initial ascending run of $\sigma$ (which may be a single point), let $\alpha(P)$ be the ascent in the rightmost occurrence
  of a 3--12 pattern in $\sigma$ whose first point is~$P$.
  Note that $\alpha$ is properly defined if and only if $\sigma$ contains \ipat{}; otherwise $\alpha(P)$ would be undefined if $P$ were the first point of~$\sigma$.

  We now define $\Psi$ to be the following length-preserving map from $\AAA_1$ to the set of bicoloured permutations:
  Define $\Psi(\sigma)$ to be the bicoloured permutation that results from moving each point $P$ from the initial ascending run of $\sigma$
  so that, for each ascent $QR$ in the image of~$\alpha$, the points in its preimage $\alpha^{-1}(QR)$ are placed in decreasing order immediately after~$R$.
  All the moved points are coloured \oo{}, as is the first point of $\Psi(\sigma)$ and non-initial points of descending runs; the unmoved non-initial points of ascending runs are coloured \ii{}.
  See Figure~\ref{figLambda} for an example; at the right, the points in the initial run are circled and the ascents in the image of $\alpha$ are indicated with a grey line.

  We first establish that $\Psi(\sigma)\in\BBB_1$.
  The colouring clearly avoids \oo-\oo{} ascents and \ii-\ii{} descents.
  If $\Psi(\sigma)$ were to contain a \oo{}--\ii{} inversion $P$--$R$, then $R$ would be the second point in an ascent $QR$, and $P$ would be an unmoved point, because a moved point would form a 3--12 with $QR$ in $\sigma$ and so be moved after $R$ by~$\Psi$.
  Thus, either
  \begin{bullets}
  \item $P$ would be the first point of $\Psi(\sigma)$, in which case $\sigma$ would contain a 43--12 pattern \mbox{$SP$--$QR$}, where $S$ could be any point in the initial ascending run of~$\sigma$, or else
  \item
  $P$ would be the second (unmoved) point of a descent~$SP$, in which case \mbox{$SP$--$QR$} would form a 43--12 in~$\sigma$.
  \end{bullets}
  Hence, $\Psi(\sigma)$ avoids \oo{}--\ii{} inversions.
  Moreover, $\Psi(\sigma)$ contains one \io{} ascent for each ascent in the image of~$\alpha$.
  Hence, $\Psi(\sigma)\in\BBB_1$ as claimed.

  Finally, we show that $\Psi$ is the inverse of the restriction of $\Lambda$ to~$\BBB_1$.

  Suppose $QR$ is in the image of~$\alpha$, and $P$ is the uppermost point of $\alpha^{-1}(QR)$.
  Then $RP$ forms a \io{} ascent in $\Psi(\sigma)$, with the remaining points of $\alpha^{-1}(QR)$ in a descending run beginning with $P$ and lying above~$R$.
  This is the only way that \io{} ascents are formed in $\Psi(\sigma)$.
  So~$\Lambda(\Psi(\sigma))=\sigma$.

  Suppose now that $\beta\in\BBB_1$ and $RP$ is a \io{} ascent in $\beta$. Then $R$ is the upper point of an ascent $QR$ in $\Lambda(\beta)$ because neither $R$ nor the first point in the ascending run containing $R$ is moved by~$\Lambda$.
  Moreover, there is no ascent below and to the right of $P$ in $\beta$, otherwise $\beta$ would contain a \oo{}--\ii{} inversion $PS$, since the second point in any ascending run is \ii{}.
  Thus $QR$ is the rightmost ascent that forms a 3--12 pattern in $\Lambda(\beta)$ with the moved point $P$, and also with any other moved points from the descending run beginning with $P$ that lie above~$R$.
  So~$\Psi(\Lambda(\beta))=\beta$.

  Hence, $\Psi$~is the partial inverse of $\Lambda$, and thus $\Lambda$ is a bijection from $\BBB$ to $\av(\text{\textup{43--12}})$, as claimed.
  By Propositions~\ref{prop322FreePerms} and~\ref{propAv1234Sym}, this is sufficient to prove that class of \thttfnl{} posets is equinumerous with $\av(\text{12--34})$.

\begin{prop}\label{propAv1234Bijection}
  The \thttf{} naturally labelled posets on $[n]$ are in bijection with permutations of length~$n$ avoiding the vincular pattern \textup{12--34}.
\end{prop}

\subsection{Enumerating 12--34-avoiding permutations}

In this section, we consider the enumeration of 12--34-avoiders, and hence also of \thttfnl{} posets.
First we present functional equations satisfied by their generating function.
Then we give a lower bound on the 
exponential term in the asymptotics
of $\av(\text{12--34})$, and
use methods of series analysis to investigate its asymptotic growth, resulting in a conjectured form for the asymptotics.

\together7
\begin{prop}\label{propAv1234FuncEq}
Suppose $H(z,x,y)$ is the trivariate generating function in which the coefficient of $z^n x^k y^\ell$ is the number of permutations of length $n$ avoiding \textup{12--34} with lowest ascent top $k$ and last entry~$\ell$.
We take $k=n+1$ for the decreasing $n$-permutation (including the permutation of length 1).
Then $H(z,x,y)=H_1(z,x,y)+H_2(z,x,y)$, where
\begin{align*}
H_1(z,x,y) & \;=\; zxy\Big( x + \frac1{1-y}\big( H_1(z,x,1) - H_1(z,x,y) + H_2(z,x,1) - H_1(z,xy,1) \big) \\
& \qquad\qquad\qquad + \frac1{1-xy}\big( H_1(z,1,xy) - H_1(z,xy,1) \big) \Big) , \\
H_2(z,x,y) & \;=\;  \frac{zy}{1-y} \big( H_1(z, x y, 1) - y H_1(z y, x, 1) + H_2(z, x y, 1) - H_2(z, x, y)  \big) .
\end{align*}
\end{prop}

\begin{proof}
  We use a generating tree approach (also sometimes known as the ECO method~\cite{BDelLPP1999}).
  See~\cite{Elizalde2007} for the enumeration of several subfamilies of 12--34 avoiding permutations using generating trees.
  We model the process of extending a 12--34-avoider by inserting a new point at the right.
  There are two cases, depending on the relative positions of the lowest ascent top and the last entry.

  If the last entry is not above the lowest ascent top, then there is no restriction on the value of the inserted point.
  This first case is enumerated by~$H_1$.

  On the other hand, if the last entry is above the lowest ascent top, then a new point cannot be inserted above the last entry, or a 12--34 would be created.
  This second case is enumerated by~$H_2$.

  \together9
  This process thus gives rise to the following transition rules, as illustrated in Figure~\ref{figAv1234}:
\[
(k,\ell) \;\longmapsto\; \begin{cases}
                           (k+1,j), & \text{\phantom{$bcde$}$(a)$ ~if $\ell\leqslant k$, for~} j=1,\,\ldots,\,\ell  , \\
                           (j,j),   & \text{\phantom{$bcde$}$(b)$ ~if $\ell\leqslant k$, for~} j=\ell+1,\,\ldots,\,k , \\
                           (k,j),   & \text{\phantom{$abde$}$(c)$ ~if $\ell\leqslant k$, for~} j=k+1,\,\ldots,\,n+1   , \\[3pt]
                           (k+1,j), & \text{\phantom{$abce$}$(d)$ ~if $\ell>k$, for~} j=1,\,\ldots,\,k   , \\
                           (k,j),   & \text{\phantom{$abcd$}$(e)$ ~if $\ell>k$, for~} j=k+1,\,\ldots,\,\ell   .
                         \end{cases}
\]
These are readily checked, as is the correctness of setting the lowest ascent top to $n+1$ for the decreasing $n$-permutation (which has no ascent).
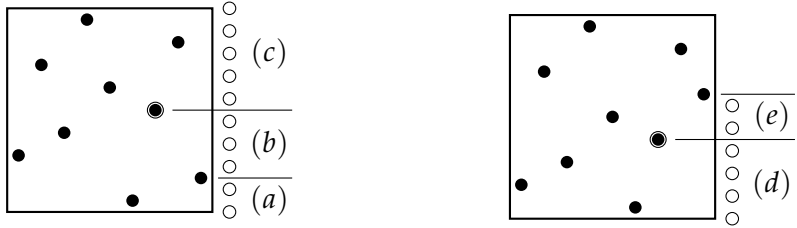
\begin{figure}[t]
  \centering
  \begin{tikzpicture}[scale=0.3]
    \plotperm{9}{3,7,4,9,6,1,5,8,2}
    \circpt{7}{5}
    \setplotptradius{0.2}
    \circpt{10.25}{.5}
    \circpt{10.25}{1.5}
    \circpt{10.25}{2.5}
    \circpt{10.25}{3.5}
    \circpt{10.25}{4.5}
    \circpt{10.25}{5.5}
    \circpt{10.25}{6.5}
    \circpt{10.25}{7.5}
    \circpt{10.25}{8.5}
    \circpt{10.25}{9.5}
    \node at (12,1) {$(a)$};
    \draw (9.75,2)--(13,2);
    \node at (12,3.5) {$(b)$};
    \draw (7.75,5)--(13,5);
    \node at (12,7.5) {$(c)$};
  \end{tikzpicture}
  \hspace{1in}
  \begin{tikzpicture}[scale=0.3]
    \plotperm{9}{2,7,3,9,5,1,4,8,6}
    \circpt{7}{4}
    \setplotptradius{0.2}
    \circpt{10.25}{.5}
    \circpt{10.25}{1.5}
    \circpt{10.25}{2.5}
    \circpt{10.25}{3.5}
    \circpt{10.25}{4.5}
    \circpt{10.25}{5.5}
    \node at (12,2) {$(d)$};
    \draw (7.75,4)--(13,4);
    \node at (12,5) {$(e)$};
    \draw (9.75,6)--(13,6);
  \end{tikzpicture}
  \caption{The possibilities for inserting a new point to the right of two 12--34-avoiders}\label{figAv1234}
\end{figure}

The process for translating an insertion encoding like this into functional equations is entirely standard.
For example, the application of rule $(a)$ to a permutation represented by the monomial $z^n x^k y^\ell$ yields permutations represented by
\[
    z^{n+1} \sum_{j=1}^\ell x^{k+1} y^j \;=\; \frac{zxy}{1-y} z^nx^k(1-y^\ell) .
\]
Extending this to a linear operator on the appropriate generating function gives
\[
\frac{zxy}{1-y} \big( H_1(z,x,1) - H_1(z,x,y) \big) ,
\]
to be included on the right of the equation for $H_1$.
Together with similar contributions from the other four rules, and a $zx^2y$ term for the initial 1-point permutation, this yields the desired functional equations.
\end{proof}

Techniques currently available would appear to be inadequate for solving these functional equations.
The primary source of their intractability is several instances of variables being present in the ``wrong slots'',
with $x$ occurring (in $xy$) in the third argument, and $y$ occurring in both the first argument (in $zy$) and the second argument (in $xy$).

Ferraz de Andrade, Lundberg and Nagle~\cite[Corollary 1.4]{FLN2015} prove that the 
exponential term in the asymptotics
of a particular \emph{subclass} of the 43--12-avoiders is
\[
\liminfty \sqrt[n]{\big|\av_n(\text{21--34}, \text{43--12})\big|/n!}\;=\;1/\log4
.
\]
Together with Proposition~\ref{propAv1234Sym}, this yields the following lower bound on the growth of 12--34-avoiders,
which improves on the bound of 0.5 
given in~\cite{CLN2013}.
\begin{prop}\label{propAv1234GRLowerBound}
  $\liminfty \sqrt[n]{\big|\av_n(\text{\textup{12--34}})\big|/n!} \;\geqs\; (\log4)^{-1} \;\approx\; 0.72134752$.
\end{prop}

The counting sequence for 12--34-avoiders --- and hence also for \thttfnl{} posets --- is
\href{https://oeis.org/A113226}{A113226} in~\cite{OEIS}.
By iterating the recurrence in the proof of Proposition~\ref{propAv1234FuncEq}, we were able to generate the first 557 terms in this sequence.
We now use the methods of series analysis, as presented by Guttmann in~\cite{Guttmann2015}, to estimate the asymptotics of $\big|\av_n(\text{\textup{12--34}})\big|$.
This approach was exploited in~\cite{CG2015,CGZJ2018} to investigate the number of permutations avoiding the classical pattern 1324; our analysis here is similar.

For each $n$, let $a_n=\big|\av_n(\text{\textup{12--34}})\big|/n!$, so
$E(z):=\sum_{n\geqs0}a_nz^n$ is
the exponential generating function for the number of 12--34-avoiders.
We consider the behaviour of the ratios
$r_n=a_{n+1}/a_n$
of consecutive terms in this sequence.

\begin{figure}[ht]
  \centering
  ~~~
  \includegraphics[width=2.6in]{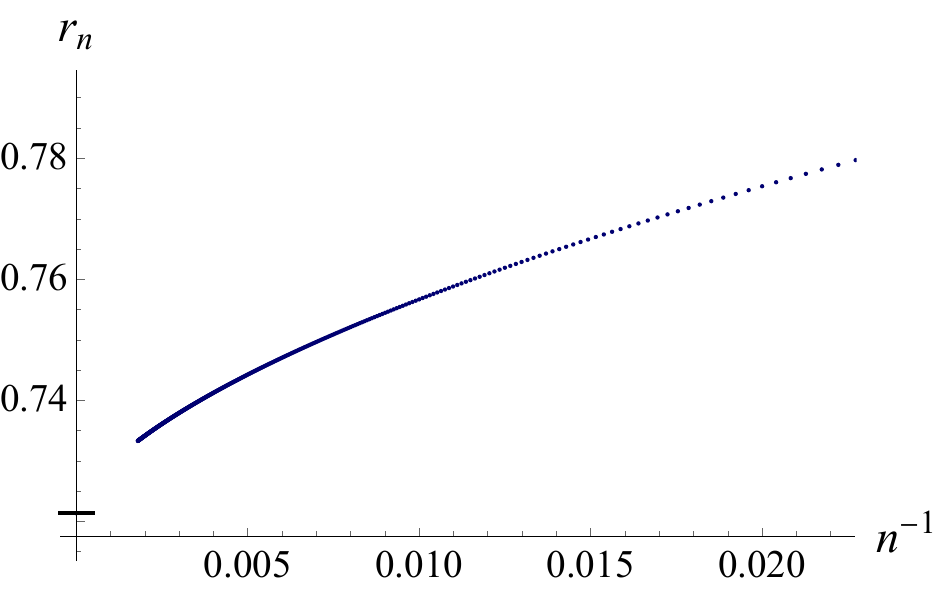}
  \hfill
  \includegraphics[width=2.6in]{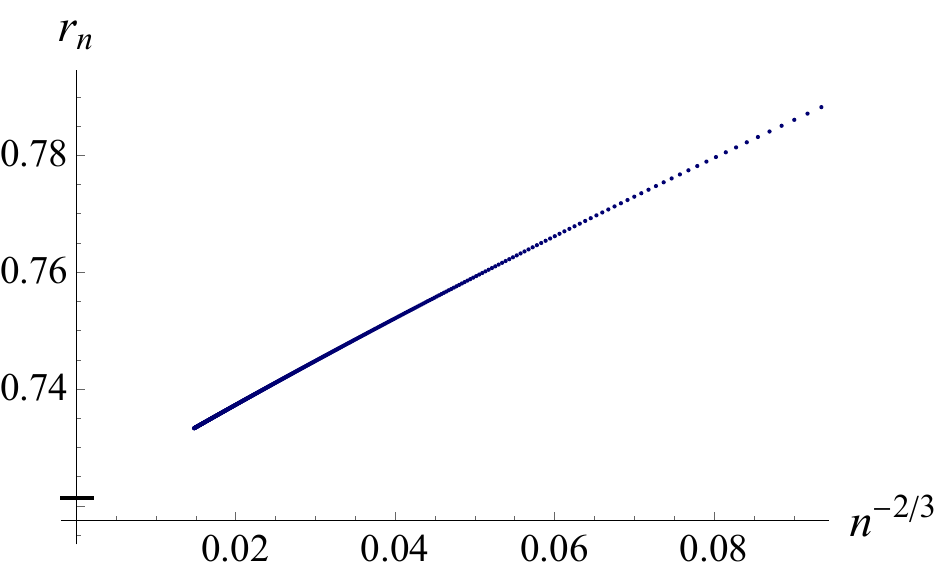}
  ~~~
  \caption{Plots of the ratios against $n^{-1}$, and against $n^{-2/3}$}\label{figPlotsRatios}
\end{figure}

If $E(z)$ were to exhibit a simple power-law singularity, with the asymptotics of the coefficients given by $a_n\sim A\cdot\gamma^n\cdot n^\beta$, for some constants $A$, $\gamma$ and $\beta$, then the ratios would satisfy
\begin{equation}\label{eqAlgSing}
  r_n \;=\; \gamma\left( 1 \:+\: \beta/n \:+\: O(n^{-2}) \right) ,
\end{equation}
in which case $r_n$ would be asymptotically linear with respect to~$n^{-1}$.
On the other hand, if the coefficients were to behave like $a_n\sim A\cdot\gamma^n\cdot\mu^{n^\alpha}\cdot n^\beta$, with a \emph{stretched exponential} factor~$\mu^{n^\alpha}$ for some constants $\mu>0$ and $\alpha<1$,
then
\[
r_n \;=\; \gamma\left( 1 \:+\: \frac{\alpha\log\mu}{n^{1-\alpha}} \:+\: \frac{\beta}n \:+\: O\big(n^{-(2-2\alpha)}\big) \right) .
\]
In particular (anticipating our findings below), when $\alpha=1/3$, this specializes to
\begin{equation}\label{eqNonAlgSing}
r_n \;=\; \gamma\left( 1 \:+\: \frac{\log\mu}{3 n^{2/3}} \:+\: \frac{\beta}n \:+\: O(n^{-4/3}) \right) ,
\end{equation}
in which $r_n$ is asymptotically linear with respect to $n^{-2/3}$.

In Figure~\ref{figPlotsRatios}, the ratios are plotted against $n^{-1}$, and against $n^{-2/3}$.
The nonlinearity of the plot against $n^{-1}$ is not consistent with the existence of an algebraic singularity, as can be seen from~\eqref{eqAlgSing}.
On the other hand, the plot against $n^{-2/3}$ appears close to being linear, which by~\eqref{eqNonAlgSing} would be consistent with the existence of a stretched exponential term with $\alpha\approx1/3$.
(We note in passing that a stretched exponential with exponent $\nfrac{1\!}3$ has recently been established rigorously for compacted binary trees~\cite{EPFW2021}.)

\begin{figure}[ht]
  \centering
  \includegraphics[width=2.6in]{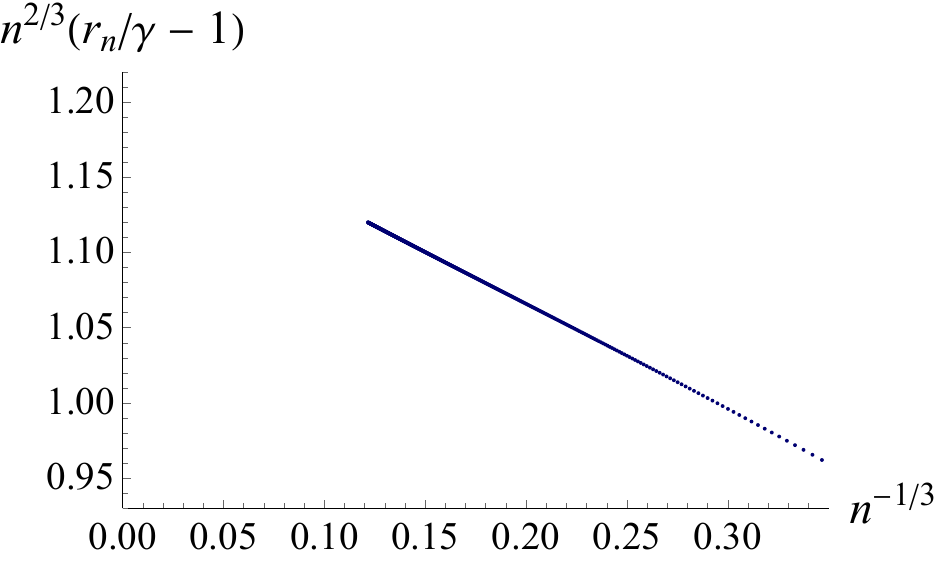}
  \caption{Plot of $n^{2/3}(r_n/\gamma-1)$ against $n^{-1/3}$, with $\gamma=(\log4)^{-1}$}\label{figPlotCubeRoot}
\end{figure}

Extrapolation of either of these plots is consistent with a limiting ratio of $(\log4)^{-1}$, as marked on the vertical axes, matching the lower bound for $\gamma$ given in Proposition~\ref{propAv1234GRLowerBound}.
Further evidence that $\gamma=(\log4)^{-1}$, and also that $\alpha=1/3$, is provided by the visual linearity of the plot in Figure~\ref{figPlotCubeRoot}, since equation~\eqref{eqNonAlgSing} implies that
\begin{equation}\label{eqCubeRoot}
n^{2/3}(r_n/\gamma-1) \;=\; \tfrac13\log\mu \:+\: \beta n^{-1/3} \:+\: O(n^{-2/3}) .
\end{equation}
Choosing a slightly different value for either $\gamma$ or $\alpha$ results in a plot with clear curvature.
Thus we believe that the 
exponential term in the asymptotics
is actually equal to its lower bound.
\begin{conj}\label{conjAv1234GR}
  $\liminfty \sqrt[n]{|\av_n(\text{\textup{12--34}})|/n!} \;=\; (\log4)^{-1}$.
\end{conj}

\begin{figure}[ht]
  \centering
  ~~~
  \includegraphics[width=2.6in]{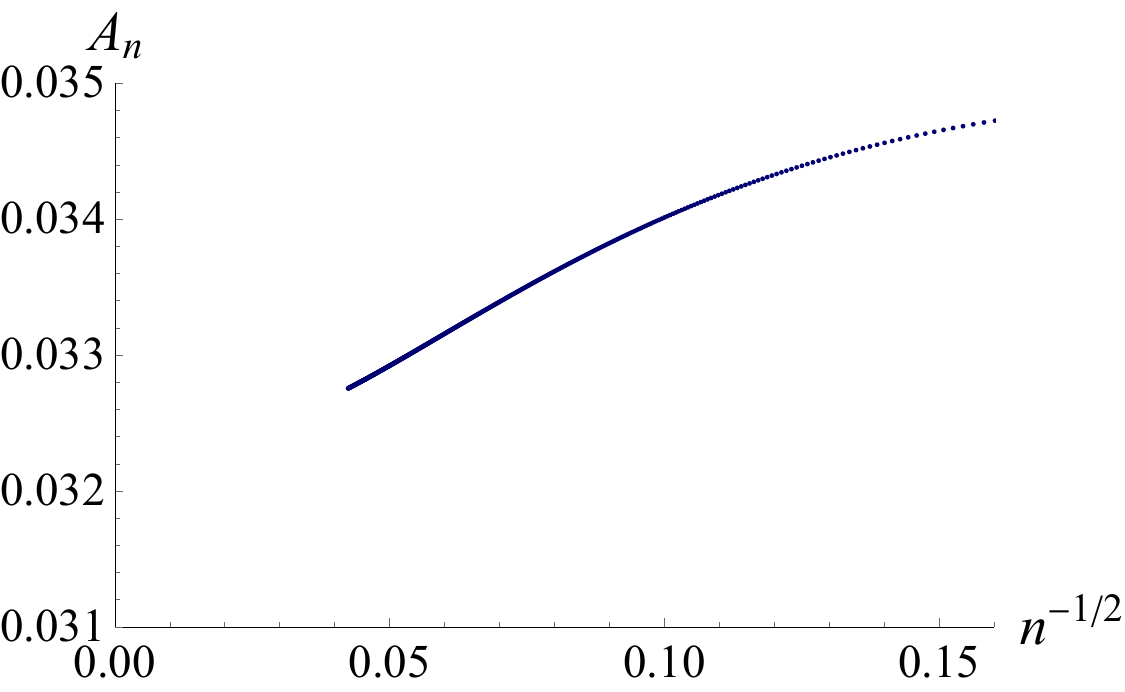}
  \hfill
  \includegraphics[width=2.6in]{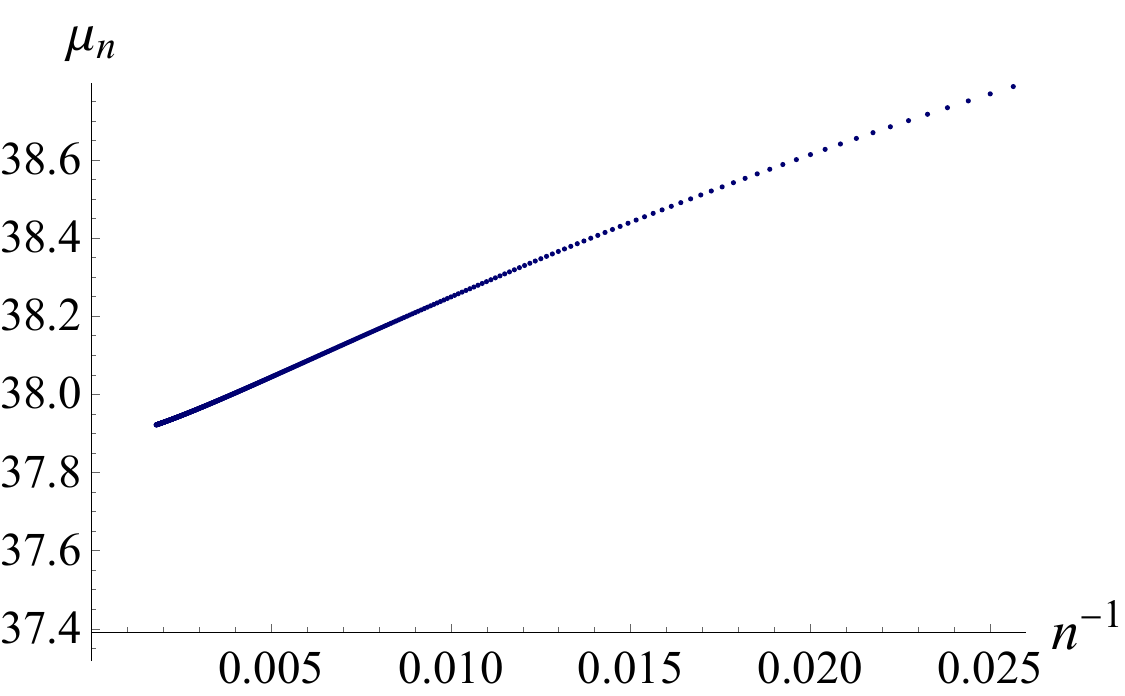}
  ~~~
\\[6pt]
  \includegraphics[width=2.6in]{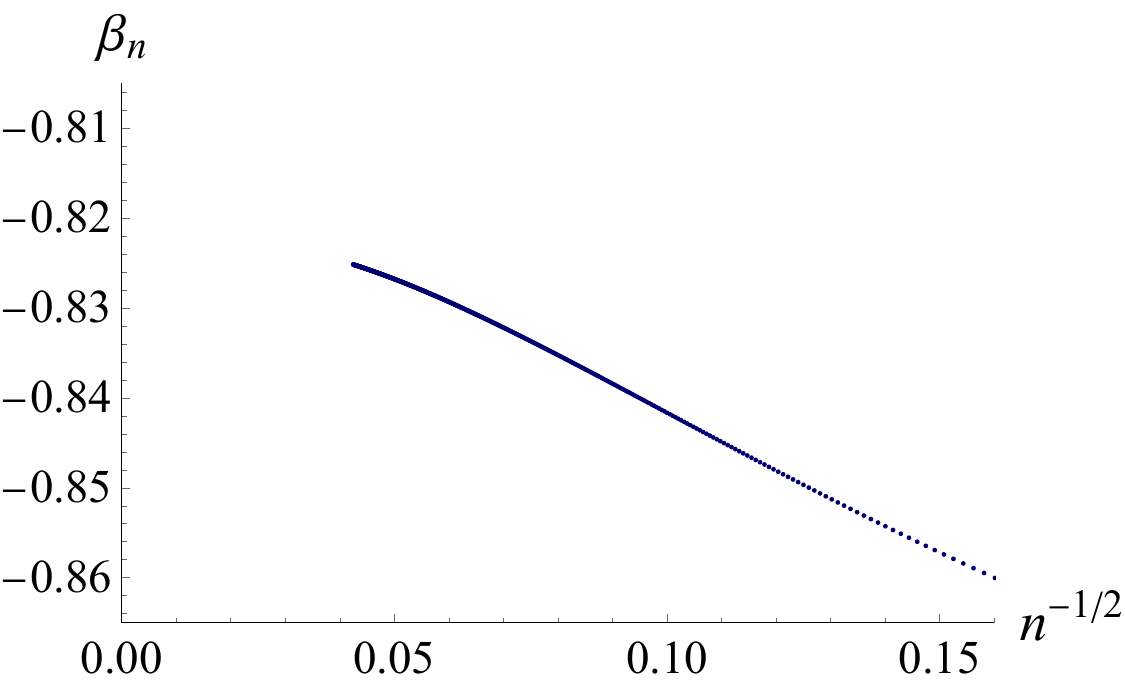}
  \caption{Plots of direct fitting estimates for $A$, $\mu$, and $\beta$}\label{figPlotsAMuBeta}
\end{figure}

By~\eqref{eqCubeRoot}, extrapolating from the plot in Figure~\ref{figPlotCubeRoot} also gives us an estimate for $\log\mu$ near to 3.6.
To gain a better approximation for $\mu$ and to approximate the values of $A$ and $\beta$, we use a direct fitting approach.
Given our assumed asymptotic form for $a_n$, we have
\begin{equation}\label{eqAv1234LogAsymptotics}
\log a_n \;\sim\; \log A \:+\: n\log\gamma \:+\: n^{1/3}\log\mu \:+\: \beta\log n .
\end{equation}
So, for each $n$, we solve the system of three linear equations
\[
 \big\{ \log A_n \:+\: k^{1/3}\log\mu_n \:+\: \beta_n\log k \;=\; \log a_k \:-\: k\log\gamma \::\: k=n,n+1,n+2 \big\} ,
\]
to give estimates $A_n$, $\mu_n$ and $\beta_n$ for $A$, $\mu$ and $\beta$, respectively.
Assuming our conjectured asymptotic form is correct, these estimates should converge to the actual values as $n$ increases.
Plots of the results are shown in Figure~\ref{figPlotsAMuBeta}, which can be extrapolated to yield approximations for the constants.

In addition, we used the \emph{Mathematica} \texttt{FindFit} function to find the values of the constants
that make the expression on the right side of~\eqref{eqAv1234LogAsymptotics}
give the best fit to different ranges of the data, as recorded in Table~\ref{tabAMuBeta}.

\begin{table}[ht]
  \centering
  \small
  \begin{tabular}{|l|c|c|c|}
    \hline
    Data & $A$ & $\mu$ & $\beta$ \\\hline
    500 terms: $\phantom{{}_0}a_{58},\ldots,a_{557}$ & $0.03351$ & $38.050$ & $-0.83314$ \\
    400 terms:              $a_{158},\ldots,a_{557}$ & $0.03312$ & $37.976$ & $-0.82876$ \\
    300 terms:              $a_{258},\ldots,a_{557}$ & $0.03298$ & $37.950$ & $-0.82710$ \\
    200 terms:              $a_{358},\ldots,a_{557}$ & $0.03286$ & $37.936$ & $-0.82618$ \\
    100 terms:              $a_{458},\ldots,a_{557}$ & $0.03280$ & $37.928$ & $-0.82557$ \\
    \hline
  \end{tabular}
  \caption{Estimates for $A$, $\mu$ and $\beta$ from \texttt{FindFit}}\label{tabAMuBeta}
\end{table}

These results show evidence of convergence and are consistent with the extrapolated intercepts from the plots in Figures~\ref{figPlotCubeRoot} and~\ref{figPlotsAMuBeta} ($\log37.9\approx3.63$).
Hence, we believe that the asymptotics have the following form.

\begin{conj}\label{conjAv1234Asymptotics}
The asymptotic number of \textup{12--34}-avoiders is given by
\[
|\av_n(\text{\textup{12--34}})| \;\sim\; A\cdot (\log4)^{-n} \cdot \mu^{n^{1/3}} \cdot n^\beta \cdot n! ,
\]
exhibiting a stretched exponential term,
with constants $A\approx0.032$, $\mu\approx 37.9$ and $\beta\approx-0.82$.
\end{conj}

\section{Discussion}\label{sectDiscussion}

Suppose $\PPP$ is an unlabelled poset with strict order relation $\prec$. Then $\PPP$ is an
\emph{interval order} if it has an \emph{interval representation}, that
is, if we can assign a real closed interval $[\ell_x,r_x]$ to each
element $x\in \PPP$ in such a way that $x\prec y$ if and only if
$r_x<\ell_y$ (so the interval corresponding to $x$ is strictly to the left of that corresponding to $y$).
The notion of an interval order was introduced by Fishburn~\cite{Fishburn1970} in 1970, who proved that $\PPP$ is an interval
order if and only if
$\PPP$ is \ttf{}.

Interval orders have attracted much attention in the literature, and
they have been shown to be equinumerous with several
combinatorial structures, namely with \emph{Fishburn matrices}, {\em
ascent sequences}, \emph{Stoimenow matchings},
and
\emph{Fishburn permutations}.
\begin{bullets}
\item A Fishburn matrix of size $n$ is an
upper-triangular $n\times n$ matrix with non-negative integer entries with
the property that every row and every column contains a nonzero
element and the total sum of the entries is $n$.

\item A sequence
$x_1x_2\ldots x_n$ of nonnegative integers is an ascent sequence if $x_1=0$ and, for each $i>1$, we have $x_i \leqs 1+\asc(x_1\ldots x_{i-1})$, where
\[
\asc(x_1\ldots x_k) \;=\; \big|\{j:1\leqs j<k \text{~and~} x_j<x_{j+1}\}\big|
\]
is the number
of \emph{ascents} in $x_1\ldots x_k$.

\item A Stoimenow
matching of size $n$ is a matching
on the set $\{1,2,\ldots,2n\}$ with no \emph{left-nestings} or \emph{right-nestings}; that is, without any pair of nested edges that have
adjacent endpoints.

\begin{figure}[t]
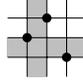

  \centering
  \mpatt{0.26}{3}{2,3,1}[1][1]
  \caption{Permutations avoiding this pattern are equinumerous with interval orders}\label{figBivincPatt}
\end{figure}
\item A Fishburn permutation is one that avoids the \emph{bivincular} pattern shown (as a \emph{mesh} pattern) in Figure~\ref{figBivincPatt}.
An occurrence in a permutation $\sigma$ of
this pattern
is an occurrence of 231 with no points in the shaded regions.
That is,
the permutation $\sigma$ contains this pattern if there exist indices $i<j$ such that $\sigma(j)+1=\sigma(i)<\sigma(i+1)$.
\end{bullets}

The enumeration of interval orders was possible only after discovering
their decomposition and linking it to ascent sequences, which were then enumerated~\cite{BMCDK2010}.
The generating function for interval
orders turns out to be the non-$D$-finite power series
\[
\sum_{n\geqs 0}\,\prod_{i=1}^{n}\,(1-(1-z)^i).
\]
Ascent sequences play a crucial role in papers \cite{DKRS2011,Jelinek2012,KR2011,Levande2013,Yan2011}
where interval orders are
enumerated explicitly with respect to extra \emph{statistics}, a statistic being a function from a set of objects to the natural numbers.
In particular, among other results, Kitaev and Remmel~\cite{KR2011} obtained
the bivariate generating function for \ttf{} posets 
counting the number of minimal elements.
As another example, a result in~\cite{DKRS2011} not only settled a conjecture of Jovovic on the number of \emph{primitive} (that is, binary) Fishburn matrices,
but also allowed the generating function to be refined to count Fishburn matrices whose entries do not exceed a fixed value~$k$,
enabling the counting of interval orders with at most $k$ indistinguishable elements.
Fishburn matrices themselves (rather than ascent sequences)
have also been used~\cite{Jelinek2012} to obtain further enumeration results related to interval orders.

These equinumerous objects, collectively known as \emph{Fishburn structures} counted by the \emph{Fishburn numbers}, are useful, not only for enumerative purposes, but
also for gaining a deeper understanding of the underlying structure of
interval orders, and for solving problems concerning them. Through
bijections we can translate properties to an
equinumerous structure that is more amenable to analysis.
Recently, Cerbai and Claesson~\cite{CC2023} introduced \emph{Fishburn trees} to this family,
obtaining simplified versions of 
some of the known bijections.


Interval orders and 
other Fishburn structures
form the base
level of a hierarchy we are about to sketch. Restricting interval
orders in two different ways and considering corresponding restrictions
on the other objects from the base level, gives two lower
levels in the hierarchy, whose existence is of interest from the perspective of bijective
combinatorics.
Analysis of these lower levels also led to the resolution of a conjecture
of Pudwell~\cite{BMCDK2010} in the theory of permutation patterns
and of a conjecture of Jovovic~\cite{DKRS2011} in the theory of
matrices.

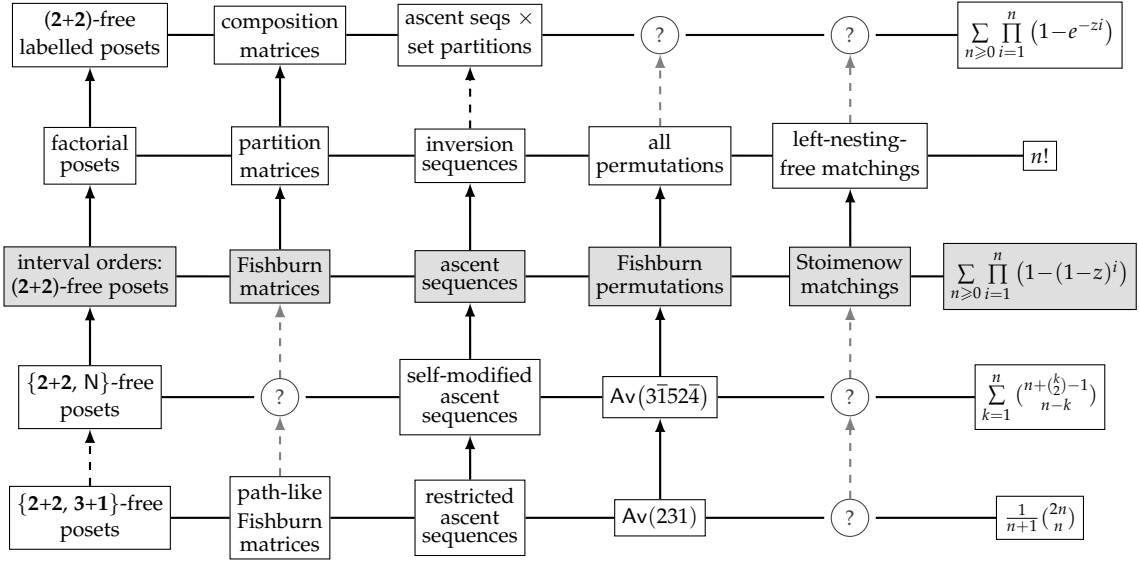
\begin{figure}[t]
\centering
\begin{tikzpicture}[xscale=2.5,yscale=1.6,>=latex]
\newcommand{\wnode}[3]{\node (#1) at (#2) [draw,rectangle] {\shortstack[c]{#3}};}
\newcommand{\gnode}[3]{\node (#1) at (#2) [draw,rectangle,fill=gray!25] {\shortstack[c]{#3}};}
\newcommand{\qnode}[2]{\node at (#2) [fill,circle,white] {???}; \node (#1) at (#2) [draw,circle,gray!50!black] {?};}
\newcommand{\oeis}[2]{\node at (#1) {\href{https://oeis.org/#2}{#2}};}
\newcommand{\karrow}[2]{\draw [->,thick] (#1)--(#2);}
\newcommand{\karrowd}[2]{\draw [->,thick,dashed] (#1)--(#2);}
\newcommand{\garrow}[2]{\draw [->,thick,gray,dashed] (#1)--(#2);}
\scriptsize
    \wnode{14}{1,4}{\ttf{}\\labelled posets}
    \wnode{24}{2,4}{composition\\[-1.5pt]matrices}
    \wnode{34}{3,4}{ascent seqs $\times$\\[-1.5pt]set partitions}
    \wnode{64}{6,4}{$\sum\limits_{n\geqs 0}\prod\limits_{i=1}^n\big(1\!-\!e^{-zi}\big)$}
    \oeis{6,3.6}{A079144}
    \wnode{13}{1,3}{factorial\\posets}
    \wnode{23}{2,3}{partition\\[-1.5pt]matrices}
    \wnode{33}{3,3}{inversion\\sequences}
    \wnode{43}{4,3}{all\\permutations}
    \wnode{53}{5,3}{left-nesting-\\[-1.5pt]free matchings}
    \wnode{63}{6,3}{$n!$}
    \oeis{6,2.75}{A000142}
    \gnode{12}{1,2}{interval orders:\\\ttf{} posets}
    \gnode{22}{2,2}{Fishburn\\matrices}
    \gnode{32}{3,2}{ascent\\sequences}
    \gnode{42}{4,2}{Fishburn\\permutations}
    \gnode{52}{5,2}{Stoimenow\\matchings}
    \gnode{62}{6,2}{$\sum\limits_{n\geqs 0}\prod\limits_{i=1}^n\big(1\!-\!(1\!-\!z)^i\big)$}
    \oeis{6,1.6}{A022493}
    \wnode{11}{1,1}{\ttf{} posets\\[-1.5pt]with maximal\\length chains}
    \wnode{31}{3,1}{self-modified\\ascent\\sequences}
    \wnode{41}{4,1}{$\av\!\big(3\overline{1}52\overline{4}\big)$}
    \wnode{61}{6,1}{$\sum\limits_{k=1}^n\binom{n+\binom{k}2-1}{n-k}$}
    \oeis{6,0.6}{A098569}
    \wnode{10}{1,0}{$\{${\bf2}+{\bf2}, \textsf{N}$\}$-free\\posets}
    \wnode{20}{2,0}{SE-free\\Fishburn\\matrices}
    \wnode{21}{2,1}{Fishburn\\matrices\\with positive\\[-1.5pt]diagonal entries}
    \wnode{30}{3,0}{$101$-avoiding\\[-1.5pt]ascent\\sequences}
    \wnode{40}{4,0}{$\av(231)$}
    \wnode{60}{6,0}{$\frac1{n+1}\binom{2n}n$}
    \oeis{6,-0.3}{A000108}
    \draw [thick] (14)--(24)--(34)--(64);
    \draw [thick] (13)--(23)--(33)--(43)--(53)--(63);
    \draw [thick] (12)--(22)--(32)--(42)--(52)--(62);
    \draw [thick] (11)--(21)--(31)--(41)--(61);
    \draw [thick] (10)--(20)--(30)--(40)--(60);
    \qnode{44}{4,4}
    \qnode{54}{5,4}
    \qnode{51}{5,1}
    \qnode{50}{5,0}
    \karrow {10}{11}    \karrow {11}{12}    \karrow {12}{13}    \karrow {13}{14}
    \karrow {20}{21}    \karrow {21}{22}    \karrow {22}{23}    \karrow {23}{24}
    \karrow {30}{31}    \karrow {31}{32}    \karrow {32}{33}    \karrowd{33}{34}
    \karrow {40}{41}    \karrow {41}{42}    \karrow {42}{43}    \garrow {43}{44}
    \garrow {50}{51}    \garrow {51}{52}    \karrow {52}{53}    \garrow {53}{54}
\end{tikzpicture}
\caption{Relationships between interval orders and associated combinatorial objects}\label{figMainHierarchy}
\end{figure}

In the other direction, Claesson and Linusson~\cite{CL2011} considered supersets of the objects from the base level that are
equinumerous to each other.
Specifically, they investigated matchings without left nestings, a superset of
Stoimenow diagrams (which are matchings with neither left or right nestings).
It turns out that there are $n!$ such matchings.
This led to the definition of a certain subset of \ttfnl{} posets counted by $n!$, which they called \emph{factorial posets}, an extension of interval orders.
Other factorial objects include permutations, \emph{inversion sequences} (a superset of ascent sequences), and \emph{partition matrices} (a superset of Fishburn matrices introduced by Claesson, Dukes and Kubitske~\cite{CDKu2011}).

One can extend this approach by considering (natural) subsets and supersets of combinatorial objects already in the hierarchy, linking bijectively structures that have the same cardinalities.
The hierarchy obtained so far in this way has twenty-one combinatorial objects on five levels, as shown in Figure~\ref{figMainHierarchy}.
At the right of each row is the enumeration of the objects in the row and a link to the entry in the OEIS~\cite{OEIS}.
Solid lines indicate known bijections, and solid arrows indicate known embeddings.
For example, the class of permutations avoiding the pattern~231,
being in one-to-one correspondence with 101-avoiding ascent sequences,
is a subset of permutations avoiding the pattern~$3\overline{1}52\overline{4}$,
which in turn is in bijection with self-modified ascent sequences,
a superset of the ascent sequences avoiding~101.
Note that $\{${\bf2}+{\bf2}, {\bf3}+{\bf1}$\}$-free posets are also counted by the Catalan numbers.
Figure~\ref{figMainHierarchy} is based on work in~\cite{BMCDK2010,CDKu2011,CL2011,DJK2011,DMcN2019,DP2010,DunSte2011,Jelinek2015}.
In particular, the bottom three levels are explained 
in~\cite{DMcN2019}.
We do not provide definitions of all the objects involved.

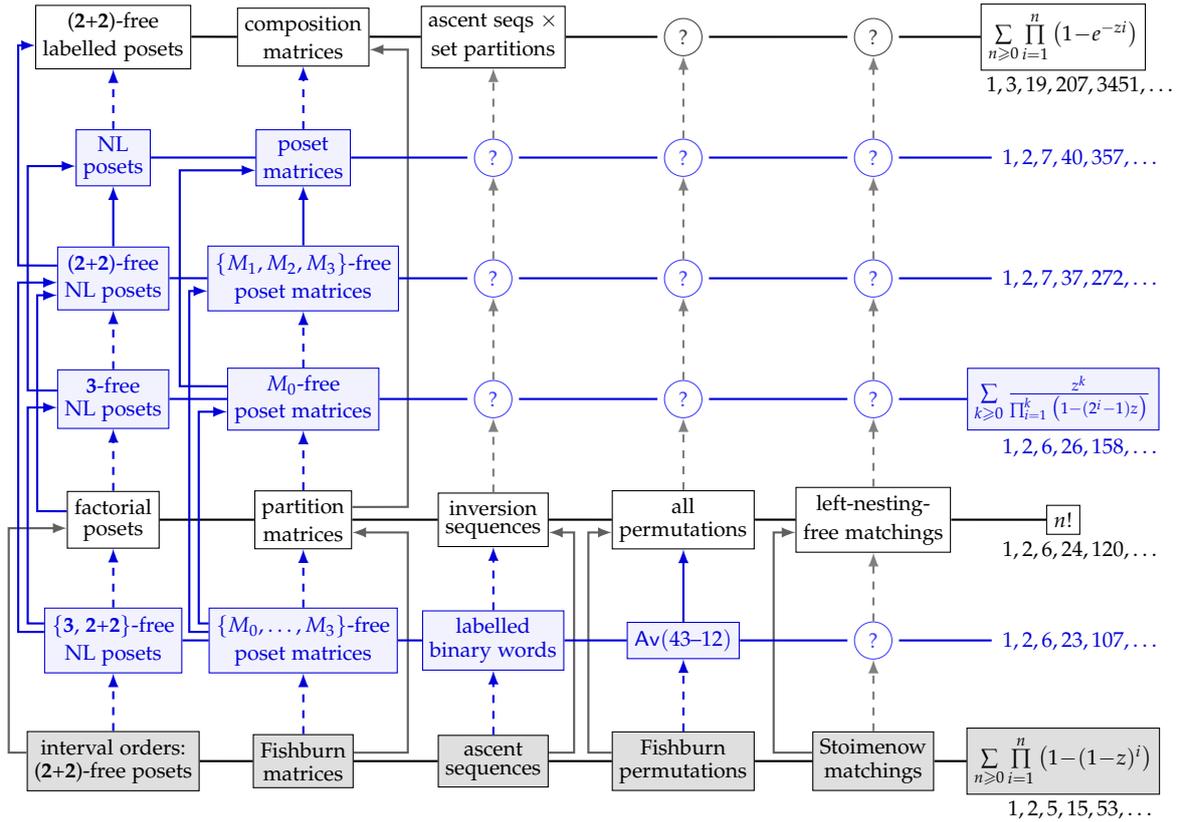
\begin{figure}[t]
\centering
\begin{tikzpicture}[xscale=2.5,yscale=1.6,>=latex]
\newcommand{\myblue}{blue!85!black}
\newcommand{\wnode}[3]{\node (#1) at (#2) [draw,rectangle] {\shortstack[c]{#3}};}
\newcommand{\gnode}[3]{\node (#1) at (#2) [draw,rectangle,fill=gray!25] {\shortstack[c]{#3}};}
\newcommand{\bnode}[3]{\node (#1) at (#2) [draw,rectangle,\myblue,fill=blue!5] {\shortstack[c]{#3}};}
\newcommand{\qnode}[2]{\node at (#2) [fill,circle,white] {???}; \node (#1) at (#2) [draw,circle,gray!50!black] {?};}
\newcommand{\bqnode}[2]{\node at (#2) [fill,circle,white] {???}; \node (#1) at (#2) [draw,circle,blue!75] {?};}
\newcommand{\kenum}[2]{\node at (#1) {$\phantom{\ldots,}#2,\ldots$};}
\newcommand{\benum}[2]{\node at (#1) [\myblue] {$\phantom{\ldots,}#2,\ldots$};}
\newcommand{\nenum}[3]{\node (#1) at (#2) [white] {$\,#3\,$}; \node at (#2) [\myblue] {$\phantom{\ldots,}#3,\ldots$};}
\newcommand{\oeis}[2]{\node at (#1) {\href{https://oeis.org/#2}{#2}};}
\newcommand{\garrow}[2]{\draw [->,thick,gray,dashed] (#1)--(#2);}
\newcommand{\barrow}[2]{\draw [->,thick,\myblue] (#1)--(#2);}
\newcommand{\barrowd}[2]{\draw [->,thick,\myblue,dashed] (#1)--(#2);}
\newcommand{\karrowl}[5]{\node[shape=coordinate] (X) at ([yshift=#4]#1 -| #3,0) {}; \draw [->,thick,gray!75!black] ([yshift=#4]#1.west) -- (X) |- ([yshift=#5]#2.west);}
\newcommand{\karrowr}[5]{\node[shape=coordinate] (X) at ([yshift=#4]#1 -| #3,0) {}; \draw [->,thick,gray!75!black] ([yshift=#4]#1.east) -- (X) |- ([yshift=#5]#2.east);}
\newcommand{\barrowl}[5]{\node[shape=coordinate] (X) at ([yshift=#4]#1 -| #3,0) {}; \draw [->,thick,\myblue] ([yshift=#4]#1.west) -- (X) |- ([yshift=#5]#2.west);}
\scriptsize
    \wnode{17}{1,7}{\ttf{}\\labelled posets}
    \wnode{27}{2,7}{composition\\[-1.5pt]matrices}
    \wnode{37}{3,7}{ascent seqs $\times$\\[-1.5pt]set partitions}
    \wnode{67}{6,7}{$\sum\limits_{n\geqs 0}\prod\limits_{i=1}^n\big(1\!-\!e^{-zi}\big)$};
    \kenum{6,6.6}{1,3,19,207,3451}
    \oeis{6,6.4}{A079144}
    \bnode{16}{1,6}{NL\\posets}
    \bnode{26}{2,6}{poset\\[-1.5pt]matrices}
    \nenum{66}{6,6}{1,2,7,40,357}
    \oeis{6,5.8}{A006455}
    \bnode{15}{1,5}{\ttf{}\\NL posets}
    \bnode{25}{2,5}{$\{M_1,M_2,M_3\}$-free\\poset matrices}
    \nenum{65}{6,5}{1,2,7,37,272}
    \oeis{6,4.8}{A367494}
    \bnode{14}{1,4}{\thf{}\\NL posets}
    \bnode{24}{2,4}{$M_0$-free\\poset matrices}
    \bnode{64}{6,4}{$\sum\limits_{k\geqs 0}\frac{z^k}{\prod_{i=1}^k\big(1-(2^i-1)z\big)}$}
    \benum{6,3.6}{1,2,6,26,158}
    \oeis{6,3.4}{A135922}
    \wnode{13}{1,3}{factorial\\NL posets}
    \wnode{23}{2,3}{partition\\[-1.5pt]matrices}
    \wnode{33}{3,3}{inversion\\sequences}
    \wnode{43}{4,3}{all\\permutations}
    \wnode{53}{5,3}{left-nesting-\\[-1.5pt]free matchings}
    \wnode{63}{6,3}{$n!$};
    \kenum{6,2.75}{1,2,6,24,120}
    \oeis{6,2.55}{A000142}
    \bnode{12}{1,2}{\thttf{}\\NL posets}
    \bnode{22}{2,2}{$\{M_0,\ldots,M_3\}$-free\\poset matrices}
    \bnode{32}{3,2}{labelled\\binary words}
    \bnode{42}{4,2}{$\av(\text{43--12})$}
    \nenum{62}{6,2}{1,2,6,23,107}
    \oeis{6,1.8}{A113226}
    \gnode{11}{1,1}{interval orders:\\\ttf{} posets}
    \gnode{21}{2,1}{Fishburn\\matrices}
    \gnode{31}{3,1}{ascent\\sequences}
    \gnode{41}{4,1}{Fishburn\\permutations}
    \gnode{51}{5,1}{Stoimenow\\matchings}
    \gnode{61}{6,1}{$\sum\limits_{n\geqs 0}\prod\limits_{i=1}^n\big(1\!-\!(1\!-\!z)^i\big)$};
    \kenum{6,0.6}{1,2,5,15,53}
    \oeis{6,0.4}{A022493}
    \karrowl{11}{13}{.45}{2}{-2}
    \karrowr{21}{23}{2.55}{2}{-3}
    \karrowr{23}{27}{2.55}{3}{-3}
    \karrowr{31}{33}{3.425}{2}{-3}
    \karrowl{41}{43}{3.5}{2}{-3}
    \karrowl{51}{53}{4.475}{2}{-3}
    \draw [thick] (17)--(27)--(37)--(67);
    \draw [thick,\myblue] (16)--(26)--(66);
    \draw [thick,\myblue] (15)--(25)--(65);
    \draw [thick,\myblue] (14)--(24)--(64);
    \draw [thick,\myblue] (12)--(22)--(32)--(42)--(62);
    \draw [thick] (13)--(23)--(33)--(43)--(53)--(63);
    \draw [thick] (11)--(21)--(31)--(41)--(51)--(61);
                      \qnode{47}{4,7}  \qnode{57}{5,7}
    \bqnode{36}{3,6} \bqnode{46}{4,6} \bqnode{56}{5,6}
    \bqnode{35}{3,5} \bqnode{45}{4,5} \bqnode{55}{5,5}
    \bqnode{34}{3,4} \bqnode{44}{4,4} \bqnode{54}{5,4}
                                      \bqnode{52}{5,2}
    \barrowd{11}{12} \barrowd{12}{13} \barrowd{13}{14} \barrowd{14}{15} \barrow{15}{16} \barrowd{16}{17}
    \barrowd{21}{22} \barrowd{22}{23} \barrowd{23}{24} \barrowd{24}{25} \barrow{25}{26} \barrowd{26}{27}
    \barrowd{31}{32} \barrowd{32}{33}  \garrow{33}{34}  \garrow{34}{35} \garrow{35}{36}  \garrow{36}{37}
    \barrowd{41}{42}  \barrow{42}{43}  \garrow{43}{44}  \garrow{44}{45} \garrow{45}{46}  \garrow{46}{47}
     \garrow{51}{52}  \garrow{52}{53}  \garrow{53}{54}  \garrow{54}{55}  \garrow{55}{56}  \garrow{56}{57}
    \barrowl{12}{14}{.55}{4}{-2}
    \barrowl{12}{15}{.5}{2}{-1}
    \barrowl{13}{15}{.6}{2}{-4}
    \barrowl{14}{16}{.55}{2}{-2}
    \barrowl{15}{17}{.5}{3}{-2}
    \barrowl{22}{24}{1.45}{4}{-3}
    \barrowl{22}{25}{1.4}{2}{-3}
    \barrowl{24}{26}{1.35}{3}{-3}
\end{tikzpicture}
\caption{The place of naturally labelled posets in the hierarchy related to interval orders}\label{figNewHierarchy}
\end{figure}


In Figure~\ref{figNewHierarchy}, we introduce to this hierarchy the objects related to \nl{} posets that we investigated in the previous sections of this paper.
These new objects are shown in blue.
As before, solid lines indicate known bijections and solid arrows indicate known embeddings.

Our bijections in Propositions~\ref{propNLMatrices}, \ref{prop3FreeMatrices} and \ref{prop322FreeMatrices} are represented by the relationships between families of \nl{} posets and families of poset matrices in the first two columns.
Similarly, the other relationships between equinumerous families of objects in the second row from the bottom represent our bijections between \thttfnl{} posets and labelled binary words in Proposition~\ref{prop322FreeWords} and with permutations avoiding 43--12 in Proposition~\ref{propAv1234Bijection}.
Note that these labelled binary words are labelled binary ascent sequences since every binary word whose first letter is~$\mathsf0$ forms an ascent sequence.
The embeddings in columns 1, 2 and 4 are all self-explanatory.

Figure~\ref{figNewHierarchy} contains thirteen explicit open embedding questions.
Out of those, we would like to highlight the following that we consider to be of particular interest:

\begin{bullets}
  \item \textbf{Rows 1 and 2}: Find an embedding of Fishburn permutations (avoiding the bivincular pattern in Figure~\ref{figBivincPatt}) into $\av(\text{43--12})$ or one of its symmetries.
  Does this translate to nice embeddings of other row 1 objects into objects in row 2?

  \item \textbf{Rows 2 and 3}: Does the trivial embedding of permutations avoiding 43--12 into the set of all permutations induce natural embeddings of the other row 2 objects into objects in row 3?
  In particular, it would be good to find an embedding of \thttfnl{} posets into factorial posets of the same size.
  A naturally labelled partial order $\prec$ on $[n]$ is \emph{factorial} if for $i, j, k\in [n]$, $i<j\prec k$ implies $i\prec k$.
  Factorial posets are \ttf{}~\cite{CL2011}.

  \item \textbf{Rows 1--3}: Discover a natural set of matchings equinumerous to \thttfnl{} posets that contain Stoimenow matchings and are left-nesting-free.

  \item \textbf{Rows 3 and 4}:
  Find an embedding of factorial posets into {\bf 3}-free \nl{} posets.
  Find an embedding of partition matrices into $M_0$-avoiding lower triangular binary matrices.
  See~\cite{CDKu2011} for the definition of a partition matrix.

  \item Investigate the relationship between Fishburn matrices and incidence matrices of \nl{} posets.
\end{bullets}

\subsection*{Acknowledgements}

The second author was partially supported by National Research Foundation of Korea (NRF) grants funded by the Korean government (MSIP) 2016R1A5A1008055 and 2019R1A2C1007518.
The third author is grateful to the SUSTech International Center for Mathematics for its hospitality during his visit to the Center in August 2024.

\bibliographystyle{plain}
{\footnotesize\bibliography{posetsbib}}

\begin{thebibliography}{10}

\bibitem{Avann1972}
S.~P. Avann.
\newblock The lattice of natural partial orders.
\newblock {\em Aequationes Math.}, 8:95--102, 1972.

\bibitem{BS2000}
Eric Babson and Einar Steingr{\'{\i}}msson.
\newblock Generalized permutation patterns and a classification of the
  {M}ahonian statistics.
\newblock {\em S\'em. Lothar. Combin.}, 44:~Article B44b, 2000.

\bibitem{BDelLPP1999}
Elena Barcucci, Alberto Del~Lungo, Elisa Pergola, and Renzo Pinzani.
\newblock E{CO}: a methodology for the enumeration of combinatorial objects.
\newblock {\em J.~Differ. Equations Appl.}, 5(4-5):435--490, 1999.

\bibitem{BMCDK2010}
Mireille Bousquet-M\'{e}lou, Anders Claesson, Mark Dukes, and Sergey Kitaev.
\newblock {$(2+2)$}-free posets, ascent sequences and pattern avoiding
  permutations.
\newblock {\em J.~Combin. Theory Ser.~A}, 117(7):884--909, 2010.

\bibitem{Branden2004}
Petter Br\"and\'en.
\newblock Counterexamples to the {N}eggers-{S}tanley conjecture.
\newblock {\em Electron. Res. Announc. Amer. Math. Soc.}, 10:155--158, 2004.

\bibitem{BPS1996}
Graham Brightwell, Hans~J\"{u}rgen Pr\"{o}mel, and Angelika Steger.
\newblock The average number of linear extensions of a partial order.
\newblock {\em J.~Combin. Theory Ser.~A}, 73(2):193--206, 1996.

\bibitem{CER2018}
Yue Cai, Richard Ehrenborg, and Margaret Readdy.
\newblock {$q$}-{S}tirling identities revisited.
\newblock {\em Electron. J. Combin.}, 25(1):~Paper 1.37, 2018.

\bibitem{CC2023}
Guilio Cerbai and Anders Claesson.
\newblock Fishburn trees.
\newblock {\em Adv. in Appl. Math.}, 151:102592, 2023.

\bibitem{CCKMM2022}
Gi-Sang Cheon, Bryan Curtis, Gukwon Kwon, and Arnauld Mesinga~Mwafise.
\newblock Riordan posets and associated incidence matrices.
\newblock {\em Linear Algebra Appl.}, 632:308--331, 2022.

\bibitem{CDKu2011}
Anders Claesson, Mark Dukes, and Martina Kubitzke.
\newblock Partition and composition matrices.
\newblock {\em J.~Combin. Theory Ser.~A}, 118(5):1624--1637, 2011.

\bibitem{CL2011}
Anders Claesson and Svante Linusson.
\newblock {$n!$} matchings, {$n!$} posets.
\newblock {\em Proc. Amer. Math. Soc.}, 139(2):435--449, 2011.

\bibitem{CG2015}
Andrew~R. Conway and Anthony~J. Guttmann.
\newblock On 1324-avoiding permutations.
\newblock {\em Adv. in Appl. Math.}, 64(0):50--69, 2015.

\bibitem{CGZJ2018}
Andrew~R. Conway, Anthony~J. Guttmann, and Paul Zinn-Justin.
\newblock 1324-avoiding permutations revisited.
\newblock {\em Adv. in Appl. Math.}, 96:312--333, 2018.

\bibitem{CLN2013}
Joshua Cooper, Erik Lundberg, and Brendan Nagle.
\newblock Generalized pattern frequency in large permutations.
\newblock {\em Electron. J. Combin.}, 20(1):~Paper 28, 2013.

\bibitem{DJK2011}
Mark Dukes, V\'{\i}t Jel\'{\i}nek, and Martina Kubitzke.
\newblock Composition matrices, {$({\bf 2}+{\bf 2})$}-free posets and their
  specializations.
\newblock {\em Electron. J. Combin.}, 18(1):~Paper 44, 2011.

\bibitem{DKRS2011}
Mark Dukes, Sergey Kitaev, Jeffrey Remmel, and Einar Steingr\'{\i}msson.
\newblock Enumerating {$(2+2)$}-free posets by indistinguishable elements.
\newblock {\em J.~Comb.}, 2(1):139--163, 2011.

\bibitem{DMcN2019}
Mark Dukes and Peter R.~W. McNamara.
\newblock Refining the bijections among ascent sequences, {$(2+2)$}-free
  posets, integer matrices and pattern-avoiding permutations.
\newblock {\em J.~Combin. Theory Ser.~A}, 167:403--430, 2019.

\bibitem{DP2010}
Mark Dukes and Robert Parviainen.
\newblock Ascent sequences and upper triangular matrices containing
  non-negative integers.
\newblock {\em Electron. J. Combin.}, 17(1):~Paper 53, 2010.

\bibitem{DunSte2011}
Paul Duncan and Einar Steingr\'{\i}msson.
\newblock Pattern avoidance in ascent sequences.
\newblock {\em Electron. J. Combin.}, 18(1):~Paper 226, 2011.

\bibitem{Elizalde2006}
Sergi Elizalde.
\newblock Asymptotic enumeration of permutations avoiding generalized patterns.
\newblock {\em Adv. in Appl. Math.}, 36(2):138--155, 2006.

\bibitem{Elizalde2007}
Sergi Elizalde.
\newblock Generating trees for permutations avoiding generalized patterns.
\newblock {\em Ann. Comb.}, 11(3-4):435--458, 2007.

\bibitem{EPFW2021}
Andrew Elvey~Price, Wenjie Fang, and Michael Wallner.
\newblock Compacted binary trees admit a stretched exponential.
\newblock {\em J.~Combin. Theory Ser.~A}, 177:105306, 2021.

\bibitem{FLN2015}
Rodrigo Ferraz~de Andrade, Erik Lundberg, and Brendan Nagle.
\newblock Asymptotics of the extremal excedance set statistic.
\newblock {\em European J. Combin.}, 46:75--88, 2015.

\bibitem{Fishburn1970}
Peter~C. Fishburn.
\newblock Intransitive indifference with unequal indifference intervals.
\newblock {\em J.~Mathematical Psychology}, 7:144--149, 1970.

\bibitem{Gessel2016}
Ira~M. Gessel.
\newblock A historical survey of {$P$}-partitions.
\newblock In {\em The mathematical legacy of {R}ichard {P}. {S}tanley}, pages
  169--188. Amer. Math. Soc., 2016.

\bibitem{Guttmann2015}
Anthony~J. Guttmann.
\newblock Analysis of series expansions for non-algebraic singularities.
\newblock {\em J.~Phys.~A}, 48(4):045209, 33, 2015.

\bibitem{Jelinek2012}
V\'{\i}t Jel\'{\i}nek.
\newblock Counting general and self-dual interval orders.
\newblock {\em J.~Combin. Theory Ser.~A}, 119(3):599--614, 2012.

\bibitem{Jelinek2015}
V\'{\i}t Jel\'{\i}nek.
\newblock Catalan pairs and {F}ishburn triples.
\newblock {\em Adv. in Appl. Math.}, 70:1--31, 2015.

\bibitem{KR2011}
Sergey Kitaev and Jeffrey Remmel.
\newblock Enumerating {$({\bf 2}+{\bf 2})$}-free posets by the number of
  minimal elements and other statistics.
\newblock {\em Discrete Appl. Math.}, 159(17):2098--2108, 2011.

\bibitem{Kreweras1981}
G.~Kreweras.
\newblock Polyn\^omes de {S}tanley et extensions lin\'eaires d'un ordre
  partiel.
\newblock {\em Math. Sci. Humaines}, (73):97--116, 1981.

\bibitem{Levande2013}
Paul Levande.
\newblock Fishburn diagrams, {F}ishburn numbers and their refined generating
  functions.
\newblock {\em J.~Combin. Theory Ser.~A}, 120(1):194--217, 2013.

\bibitem{Neggers1978}
Joseph Neggers.
\newblock Representations of finite partially ordered sets.
\newblock {\em J.~Combin. Inform. System Sci.}, 3(3):113--133, 1978.

\bibitem{OEIS}
The OEIS Foundation~Inc.
\newblock The {O}n-{L}ine {E}ncyclopedia of {I}nteger {S}equences.
\newblock Published electronically at
  \href{https://oeis.org}{https:/$\!$/oeis.org}.

\bibitem{Stanley1972}
Richard~P. Stanley.
\newblock {\em Ordered structures and partitions}.
\newblock American Mathematical Society, 1972.

\bibitem{Stanley1973}
Richard~P. Stanley.
\newblock Acyclic orientations of graphs.
\newblock {\em Discrete Math.}, 5:171--178, 1973.

\bibitem{StanleyEC1}
Richard~P. Stanley.
\newblock {\em Enumerative Combinatorics. {V}ol. 1}.
\newblock Cambridge University Press, 1997.

\bibitem{Stanley1997}
Richard~P. Stanley.
\newblock Problem 10572.
\newblock {\em The American Mathematical Monthly}, 104(2):168, 1997.

\bibitem{SL1999}
Richard~P. Stanley and Stephen~C. Locke.
\newblock Graphs without increasing paths: Solution to {P}roblem 10572.
\newblock {\em The American Mathematical Monthly}, 106(2):168, 1999.

\bibitem{Stembridge1997}
John~R. Stembridge.
\newblock Enriched {$P$}-partitions.
\newblock {\em Trans. Amer. Math. Soc.}, 349(2):763--788, 1997.

\bibitem{Stembridge2007}
John~R. Stembridge.
\newblock Counterexamples to the poset conjectures of {N}eggers, {S}tanley, and
  {S}tembridge.
\newblock {\em Trans. Amer. Math. Soc.}, 359(3):1115--1128, 2007.

\bibitem{Yan2011}
Sherry H.~F. Yan.
\newblock On a conjecture about enumerating {$(2+2)$}-free posets.
\newblock {\em European J. Combin.}, 32(2):282--287, 2011.

\end{thebibliography}

\end{document}